\documentclass[10pt]{amsart}

\usepackage{amssymb,amscd,amsthm,amsxtra}
\usepackage{latexsym}
\usepackage{dsfont,latexsym, color}
\usepackage{mathrsfs}
\usepackage{lmodern}
\usepackage{mathrsfs}
\usepackage{fancyhdr}

\usepackage{color}
\definecolor{citation}{rgb}{0.2,0.58,0.2} 
\definecolor{formula}{rgb}{0.1,0.2,0.6}
\definecolor{url}{rgb}{0.3,0,0.5} 
\usepackage[colorlinks=true,linkcolor=formula,urlcolor=url,citecolor=citation]{hyperref} 
\def\vs{\vspace{1mm}}

\def\dx{\,{\rm d}x}

\def\dy{\,{\rm d}y}

\def\dist{\,{\rm dist}}

\allowdisplaybreaks
\makeatletter
\DeclareRobustCommand*{\bfseries}{%
	\not@math@alphabet\bfseries\mathbf
	\fontseries\bfdefault\selectfont
	\boldmath
}

\makeatother

\newlength{\defbaselineskip}
\newcommand{\setlinespacing}[1]
{\setlength{\baselineskip}{#1 \defbaselineskip}}
\newcommand{\dashint}{\mathop{\int\hskip -1,05em -\, \!\!\!}\nolimits}

\def\Xint#1{\mathchoice
	{\XXint\displaystyle\textstyle{#1}}%
	{\XXint\textstyle\scriptstyle{#1}}%
	{\XXint\scriptstyle\scriptscriptstyle{#1}}%
	{\XXint\scriptscriptstyle\scriptscriptstyle{#1}} %
	\!\int}
\def\XXint#1#2#3{{\setbox0=\hbox{$#1{#2#3}{\int}$}
		\vcenter{\hbox{$#2#3$}}\kern-.5\wd0}}

\def\dashint{\Xint-}

\newtheorem{theorem}{Theorem}[section]

\newtheorem{lemma}{Lemma}[section]
\newtheorem{proposition}{Proposition}[section]
\numberwithin{equation}{section}
\allowdisplaybreaks[4]

\def\loc{\operatorname{loc}}

\newcommand{\kk}{\kappa}

\def\er{\mathds R}

\newcommand{\ern}{\mathds{R}^n}

\newcommand{\ti}[1]{\tilde{#1}}
\newcommand{\erN}{\mathds{R}^N}

\newcommand\eps\varepsilon
\newcommand{\ernN}{\mathds{R}^{N\times n}}

\newcommand{\data}{\textnormal{\texttt{data}}}

\newcommand{\be}{\begin{equation}}
\newcommand{\ee}{\end{equation}}

\newcommand{\rr}{\rho}
\newcommand{\modulo}[1]{{\left|#1\right|}}
\newcommand{\norma}[1]{{\left\|#1\right\|}}

\newcommand{\snr}[1]{\lvert #1\rvert}

\newcommand{\nr}[1]{\lVert #1 \rVert}

\newcommand{\sss}{\varsigma}

\newcommand{\F}{\varphi}

\newcommand{\dd}{D}

\newcommand\ap{``}
\title[Partial regularity for quasiconvex integrals]{A limiting case in partial regularity \\for quasiconvex functionals}

\author[M. Piccinini]{Mirco Piccinini}
\address[Mirco Piccinini]{Dipartimento di Matematica e Informatica, Universit\`a degli Studi di Parma, Campus - Parco Area delle Scienze, 53/a, 43124 Parma, Italy}
\email{\href{mailto:mirco.piccinini@unipr.it}{mirco.piccinini@unipr.it} }

\subjclass[2010]{35J60, 49J45, 49N60\vspace{1mm}} 

\keywords{Regularity, Quasiconvex functionals, Degenerate variational integrals\vspace{1mm}}

\thanks{{\it Acknowledgements.}\  The author is supported by INdAM Projects ``Fenomeni non locali in problemi locali",\!\char`_CUP\_E55F22000270001 and ``Problemi non locali: teoria cinetica e non uniforme ellitticit\`a'', CUP\_E53C220019320001, and also by the Project ``Local vs Nonlocal: mixed type operators and nonuniform ellipticity", \!\char`_CUP\_D91B21005370003.\\
		The authors is grateful to Cristiana De Filippis for suggesting the problem and for her fruitful observations and advices which helped improving the quality of the manuscript.
	\vspace{1mm}}

\begin{document}

\begin{abstract}
Local minimizers of nonhomogeneous quasiconvex variational integrals with standard $p$-growth of the type
$$
w\mapsto \int \left[F(Dw)-f\cdot w\right]\dx
$$
feature almost everywhere $\mbox{BMO}$-regular gradient provided that $f$ belongs to the borderline Marcinkiewicz space $L(n,\infty)$.
\vspace{3mm}
{\it
	\begin{center}
		Dedicated to Giuseppe Mingione \\ on the occasion of his $50^\mathrm{th}$ birthday, with admiration.
	\end{center}
}

\end{abstract}

\maketitle

{\footnotesize\setlinespacing{0.8}
	\setcounter{tocdepth}{2}
	\tableofcontents
}

\section{Introduction}\label{si}
In this paper we provide a limiting partial regularity criterion for vector-valued minimizers~$u : \Omega\subset \mathds{R}^{n} \to \erN$, $n\ge 2$, $N>1$, of nonhomogeneous, quasiconvex variational integrals as:
\begin{equation}\label{main_fnc}
	W^{1,p}(\Omega;\erN) \ni w \mapsto \mathcal{F}(w;\Omega):= \int_{\Omega}[F(D w) -f \cdot w ]\dx,
\end{equation}
with standard $p$-growth. More precisely, we infer the optimal \cite[Section 9]{kumig} $\varepsilon$-regularity condition
\begin{flalign*}
&\sup_{B_{\rho}\Subset \Omega}\rho^{m}\dashint_{B_{\rho}}\snr{f}^{m}\dx\lesssim \varepsilon\nonumber \\
&\qquad \qquad \qquad\qquad\quad   \Longrightarrow \ Du \ \mbox{has a.e. bounded mean oscillation},
\end{flalign*}
and the related borderline function space criterion
$$
f\in L(n,\infty) \ \Longrightarrow \ \sup_{B_{\rho}\Subset \Omega}\rho^{m}\dashint_{B_{\rho}}\snr{f}^{m}\dx\lesssim \varepsilon.
$$
This is the content of our main theorem.
\begin{theorem}\label{t1}
	Under assumptions \eqref{assf}$_{1,2,3}$, \eqref{sqc} and \eqref{f}, let $u\in W^{1,p}(\Omega,\mathds{R}^{N})$ be a local minimizer of functional \eqref{main_fnc}. Then, there exists a number~$\eps_*\equiv \eps_*(\data)>0$ such that if
	\begin{equation}\label{est_Lninfty}
     \|f\|_{L^{n,\infty}(\Omega)} \leqslant\left(\frac{\snr{B_1}}{4^{n/m}}\right)^{1/n}\eps_*,
	\end{equation}
    then there exists an open set $\Omega_u \subset \Omega$ with~$\snr{\Omega \setminus \Omega_u}=0$  such that 
	\begin{equation}\label{t1 1}
    \dd u \in BMO_{\loc}(\Omega_u; \ernN).
	\end{equation}
	Moreover, the set~$\Omega_u$ can be characterized as follows
	\begin{eqnarray}\label{t1 2}
	\Omega_u &:=& \left\{    x_0 \in \Omega : \exists \eps_{x_0}, \rr_{x_0} >0\right. \\*
	&& \qquad\qquad \left.\text{such that}  \  \mathcal{E}(u;B_\rr(x_0)) \leqslant\eps_{x_0} \ \text{for some} \ \rr \leqslant\rr_{x_0} \right\},\notag
	\end{eqnarray}
	where~$\mathcal{E}(\cdot)$ is the usual excess functional defined as
 \begin{equation}\label{eccess}
	\mathcal{E}(w,z_0;B_\rr(x_0)) := \left( \ \dashint_{B_\rr(x_0)} \snr{z_0}^{p-2} \snr{\dd w -z_0}^2 + \snr{\dd w - z_0}^p \ \dx \right)^\frac{1}{p}.
\end{equation}
 
\end{theorem}
We immediately refer to Section \ref{sa} below for a description of the structural assumptions in force in Theorem \ref{t1}. Let us put our result in the context of the available literature. The notion of quasiconvexity was introduced by Morrey~\cite{Mor57} in relation to the delicate issue of semicontinuity of multiple integrals in Sobolev spaces: an integrand $F(\cdot)$  is a {\it quasiconvex} whenever 
\begin{flalign}\label{qc}
	\dashint_{B_{1}(0)}F(z+D\varphi) \dx \geqslant F(z)\quad \mbox{holds for all} \ \ z\in \mathds{R}^{N\times n}, \ \ \varphi\in C^{\infty}_{\rm c}(B_{1}(0),\mathds{R}^{N}).
\end{flalign}
 Under power growth conditions, \eqref{qc} is proven to be necessary and sufficient for the sequential weak lower semicontinuity on~$W^{1,p}(\Omega;\erN)$; see~\cite{AF84,BM84, ma3, ma2, Mor57}. It is worth stressing that quasiconvexity is a strict generalization of convexity: the two concepts coincide in the scalar setting ($N=1$), or for $1$-d problems ($n=1$), but sharply differ in the multidimensional case: every convex function is quasiconvex thanks to Jensen's inequality, while the determinant is quasiconvex (actually polyconvex), but not convex, cf. \cite[Section 5.1]{giu}. Another distinctive trait is the nonlocal nature of quasiconvexity: Morrey \cite{Mor57} conjectured that there is no condition involving only $F(\cdot)$ and a finite number of its derivatives that is both necessary and sufficient for quasiconvexity, fact later on confirmed by Kristensen \cite{k}. A peculiarity of quasiconvex functionals is that minima and critical points (i.e., solutions to the associated Euler-Lagrange system) might have very different behavior under the (partial) regularity viewpoint. In fact, a classical result of Evans \cite{Eva86} states that the gradient of minima is locally H\"older continuous outside a negligible, \ap singular" set, while a celebrated counterexample due to M\"uller \& \v{S}ver\'ak \cite{musv} shows that the gradient of critical points may be everywhere discontinuous. After Evans seminal contribution \cite{Eva86}, the partial regularity theory was extended by Acerbi \& Fusco \cite{AF87} to possibly degenerate quasiconvex functionals with superquadratic growth, and by Carozza \& Fusco \& Mingione \cite{CFM98} to subquadratic, nonsingular variational integrals. A unified approach that allows simultaneously handling degenerate/nondegenerate, and singular/nonsingular problems, based on the combination of $\mathcal{A}$-harmonic approximation \cite{dust}, and $p$-harmonic approximation \cite{dumi1}, was eventually proposed by Duzaar \& Mingione \cite{dumi}. Moreover, Kristensen \& Mingione \cite{KM07} proved that the Hausdorff dimension of the singular set of Lipschitz continuous minimizers of quasiconvex multiple integrals is strictly less than the ambient space dimension $n$, see also \cite{bgik} for further developments in this direction. We refer to \cite{AM01,dlsv,gme,gm1,gk,gkpq,ma5,Sch08,schm} for an (incomplete) account of classical, and more recent advances in the field. In all the aforementioned papers are considered homogeneous functionals, i.e. $f\equiv 0$ in \eqref{main_fnc}. The first sharp $\varepsilon$-regularity criteria for nonhomogeneous quasiconvex variational integrals guaranteeing almost everywhere gradient continuity under optimal assumptions on $f$ were obtained by De Filippis \cite{dqc}, and De Filippis \& Stroffolini \cite{ds}, by connecting the classical partial regularity theory for quasiconvex functionals with nonlinear potential theory for degenerate/singular elliptic equations, first applied in the context of partial regularity for strongly elliptic systems by Kuusi \& Mingione \cite{kumi}. Potential theory for nonlinear PDE originates from the classical problem of determining the best condition on $f$ implying gradient continuity in the Poisson equation $-\Delta u=f$, that turns out to be formulated in terms of the uniform decay to zero of the Riesz potential, in turn implied by the membership of $f$ to the Lorentz space $L(n,1)$, \cite{CiGA,kumig}. In this respect, a breakthrough result due to Kuusi \& Mingione \cite{kumil,kumi0} states that the same is true for the nonhomogeous, degenerate $p$-Laplace equation - in other words, the regularity theory for the nonhomogeneous $p$-Laplace PDE coincides with that of the Poisson equation up to the $C^{1}$-level. This important result also holds in the case of singular equations \cite{dz,tn}, for general, uniformly elliptic equations \cite{ba2}, up to the boundary \cite{cm,cm1}, and at the level of partial regularity for $p$-Laplacian type systems without Uhlenbeck structure, \cite{by,kumi}. We conclude by highlighting that our Theorem \ref{t1} fits this line of research as, it determines for the first time in the literature optimal conditions on the inhomogeneity $f$ assuring partial $\mbox{BMO}$-regularity for minima of quasiconvex functionals expressed in terms of the limiting function space $L(n,\infty)$. 
 \subsubsection*{Outline of the paper} In Section~\ref{preliminaries} we recall some well-known results from the study of nonlinear problems also establishing some Caccioppoli and Gehring type lemmas. In Section~\ref{excess_decay} we prove the excess decay estimates;  considering separately the nondegenerate and the degenerate case. Section~\ref{proof_of_the_main} is devoted to the proof of Theorem~\ref{t1}.

 \subsection{Structural assumptions}\label{sa}
 In \eqref{main_fnc}, the integrand $F\colon \mathds{R}^{N\times n}\to \mathds{R}$ satisfies
 \begin{flalign}\label{assf}
	\begin{cases}
		\ F\in C^{2}_{\loc}(\mathds{R}^{N\times n})\\*[0.5ex]
		\ \Lambda^{-1}\snr{z}^{p}\leqslant F(z)\leqslant\Lambda \snr{z}^{p}\\*[0.5ex]
		\ \snr{\partial^{2}F(z)}\leqslant\Lambda \snr{z}^{p-2}\\*[0.5ex]
		\ \snr{\partial^{2} F(z_{1})-\partial^{2}F(z_{2})}\le\mu\left(\displaystyle\frac{\snr{z_{2}-z_{1}}}{\snr{z_{2}}+\snr{z_{1}}}\right) \left(\snr{z_{1}}^{2}+\snr{z_{2}}^{2}\right)^{\frac{p-2}{2}}
	\end{cases}
\end{flalign}
for all~$z\in \mathds{R}^{N\times n}$,~$\Lambda\geqslant 1$ being a positive absolute constant and~$\mu\colon [0,\infty)\to [0,1]$ being a concave nondecreasing function with $\mu(0)=0$.  In the rest of the paper we will always assume $p\ge 2$. In order to derive meaningful regularity results, we need to update \eqref{qc} to the stronger strict quasiconvexity condition
\begin{flalign}\label{sqc}
	\int_{B}\left[F(z+D\varphi)-F(z)\right]\  \dx\geqslant \lambda\int_{B}(\snr{z}^{2}+\snr{D\varphi}^{2})^{\frac{p-2}{2}}\snr{D\varphi}^{2} \dx, 
\end{flalign}
holding for all~$z\in \mathds{R}^{N\times n}$ and~$\varphi\in W^{1,p}_{0}(B,\mathds{R}^{N})$, with~$\lambda$ being a positive, absolute constant. Furthermore, we allow the integrand~$F(\cdot)$ to be degenerate elliptic in the origin. More specifically, we assume that~$F(\cdot)$ features degeneracy of $p$-Laplacian type at the origin, i.~\!e.
\begin{flalign}\label{p0}
	\left| \ \frac{\partial F(z)-\partial F(0)-\snr{z}^{p-2}z}{\snr{z}^{p-1}} \ \right|\to 0\, \qquad \text{as}~\snr{z}\to 0\,,
\end{flalign}
which means that we can find a function~$\omega\colon (0,\infty)\to (0,\infty)$ such that
\begin{eqnarray}\label{p0.1}
	\snr{z}\leqslant\omega(s) \ \Longrightarrow \ \snr{\partial F(z)-\partial F(0)-\snr{z}^{p-2}z}\leqslant s\snr{z}^{p-1},
\end{eqnarray}
for every~$z\in \mathds{R}^{N\times n}$ and all~$s\in (0,\infty)$. Moreover, the right-hand side term~$f\colon \Omega\to \mathds{R}^{N}$ in~\eqref{main_fnc} verifies as minimal integrability condition the following
\begin{eqnarray}\label{f}
	f\in L^{m}(\Omega,\mathds{R}^{N})\quad \mbox{with} \ \ 2>m>\begin{cases}
		\ 2n/(n+2)\quad &\mbox{if} \ \ n>2\\*[0.5ex]
		\ 3/2\quad &\mbox{if} \ \ n=2,
	\end{cases}
\end{eqnarray}
which, being~$p\geqslant 2$, in turn implies that
\begin{eqnarray}\label{f.0}
	f\in W^{1,p}(\Omega,\mathds{R}^{N})^{*}\qquad \mbox{and}\qquad m'<2^{*}\leqslant p^{*}.
\end{eqnarray}
Here it is intended that, when~$p\geqslant n$, the Sobolev conjugate exponent~$p^{*}$ can be chosen as large as needed - in particular it will always be larger than~$p$. By~\eqref{qc} and~$\eqref{assf}_{2}$ we have
\begin{eqnarray}\label{df}
	\snr{\partial F(z)}\leqslant c\snr{z}^{p-1},
\end{eqnarray}
with~$c\equiv c(n,N, \Lambda,p)$; see for example~\cite[proof of Theorem 2.1]{ma3}. Finally,~\eqref{sqc} yields that for all~$z\in \mathds{R}^{N\times n}$, $\xi\in \mathds{R}^{N}$, $\zeta\in \mathds{R}^{n}$ it is
\begin{eqnarray}\label{sqc.1}
	\partial^{2}F(z)\langle\xi\otimes \zeta,\xi\otimes \zeta\rangle\geqslant 2\lambda\snr{z}^{p-2}\snr{\xi}^{2}\snr{\zeta}^{2},
\end{eqnarray}
see~\cite[Chapter 5]{giu}.

\section{Preliminaries}\label{preliminaries}
In this section we display our notation and collect some basic results that will be helpful later on.
\subsection{Notation}
In this paper, $\Omega\subset \er^n$ is an open, bounded domain with Lipschitz boundary, and $n \geqslant 2$. By $c$ we will always denote a general constant larger than one, possibly depending on the data of the problem. Special occurrences will be denoted by $c_*,  \tilde c$ or likewise. Noteworthy dependencies on parameters will be highlighted by putting them in parentheses. Moreover, to simplify the notation, we shall array the main parameters governing functional \eqref{main_fnc} in the shorthand $\textnormal{\texttt{data}}:=\left(n,N,\lambda,\Lambda,p,\mu(\cdot),\omega(\cdot)\right)$. By $ B_r(x_0):= \{x \in \er^n  : |x-x_0|< r\}$, we denote the open ball with radius $r$, centred at $x_{0}$; when not necessary or clear from the context, we shall omit denoting the center, i.e. $B_{r}(x_{0})\equiv B_{r}$ - this will happen, for instance, when dealing with concentric balls. For $x_{0}\in \Omega$, we abbreviate $d_{x_{0}}:=\min\left\{1,\dist(x_{0},\partial \Omega)\right\}$. Moreover, with $ B \subset \er^{n}$ being a measurable set with bounded positive Lebesgue measure $0<| B|<\infty$, and $a \colon  B \to \er^{k}$, $k\geqslant 1$, being a measurable map, we denote $$
(a)_{ B} \equiv \dashint_{ B}  a(x)\  \dx  :=  \frac{1}{| B|}\int_{B}  a(x)\dx.
$$
We will often employ the almost minimality property of the average, i.e.
\begin{flalign}\label{minav}
	\left( \ \dashint_{B}\snr{a-(a)_{B}}^{t} \dx\right)^{1/t}\leqslant2\left( \ \dashint_{B}\snr{a-z}^{t} \dx\right)^{1/t}
\end{flalign}
for all $z\in \mathds{R}^{N\times n}$ and any $t\geqslant 1$. Finally, if $t>1$ we will indicate its conjugate by $t':=t/(t-1)$ and its Sobolev exponents as $t^{*}:=nt/(n-t)$ if $t<n$ or any number larger than one for $t\geqslant n$ and $t_{*}:=\max\left\{nt/(n+t),1\right\}$. 

\subsection{Tools for nonlinear problems} When dealing with $p$-Laplacian type problems, we shall often use the auxiliary vector field $V_{s}\colon \er^{N\times n} \to  \er^{N\times n}$, defined by
\begin{flalign*}
	V_{s}(z):= (s^{2}+|z|^{2})^{(p-2)/4}z\qquad \text{with}~p\in (1,\infty), \ \ s\geqslant 0, \ \ z\in \mathds{R}^{N\times n},
\end{flalign*}
incorporating the scaling features of the $p$-Laplacian. If $s=0$ we simply write $V_{s}(\cdot)\equiv V(\cdot)$. A couple of useful related inequalities are
\begin{flalign}\label{Vm}
	\begin{cases}
		\ \snr{V_{s}(z_{1})-V_{s}(z_{2})}\approx (s^{2}+\snr{z_{1}}^{2}+\snr{z_{2}}^{2})^{(p-2)/4}\snr{z_{1}-z_{2}}\\*[0.5ex]
		\ \snr{V_{s}(z_{1}+z_{2})}\lesssim \snr{V_{s}(z_{1})}+\snr{V_{s}(z_{2})}\\*[0.5ex]
		\  \snr{V_{s_1}(z)} \approx \snr{V_{s_2}(z)}, \ \mbox{if}\ \frac{1}{2}s_2 \leqslant s_1 \leqslant2 s_2\\*[0.5ex]
		\  \snr{V(z_1)-V(z_2)}^2 \approx \snr{V_{\snr{z_1}}(z_1-z_2)}^2, \ \mbox{if}\ \frac{1}{2}\snr{z_2} \leqslant\snr{z_1} \leqslant2 \snr{z_2}
	\end{cases}
\end{flalign}
and
\begin{eqnarray}\label{equiv.1}
	\snr{V_{s}(z)}^{2}\approx s^{p-2}\snr{z}^{2}+\snr{z}^{p}\qquad \mbox{with} \ \ p\geqslant 2,
\end{eqnarray}
where the constants implicit in \ap $\lesssim$", \ap $\approx$" depend on $n,N,p$. A relevant property which is relevant for the nonlinear setting is recorded in the following lemma.
\begin{lemma}\label{l6}
	Let $t>-1$, $s\in [0,1]$ and $z_{1},z_{2}\in \mathds{R}^{N\times n}$ be such that $s+\snr{z_{1}}+\snr{z_{2}}>0$. Then
	\begin{flalign*}
		\int_{0}^{1}\left[s^2+\snr{z_{1}+y(z_{2}-z_{1})}^{2}\right]^{\frac{t}{2}} \ \dy\approx (s^2+\snr{z_{1}}^{2}+\snr{z_{2}}^{2})^{\frac{t}{2}},
	\end{flalign*}
	with constants implicit in "$\approx$" depending only on $n,N,t$.
\end{lemma}
The following iteration lemma will be helpful throughout the rest of the paper; for a proof we refer the reader to \cite[Lemma 6.1]{giu}.
\begin{lemma}\label{l5}
	Let $h\colon [\rr_{0},\rr_{1}]\to \mathds{R}$ be a non-negative and bounded function, and let $\theta \in (0,1)$, $A,B,\gamma_{1},\gamma_{2}\geqslant 0$ be numbers. Assume that $h(t)\leqslant\theta h(s)+A(s-t)^{-\gamma_{1}}+B(s-t)^{-\gamma_{2}}$ holds for all $\rr_{0}\leqslant t<s\leqslant\rr_{1}$. Then the following inequality holds $h(\rr_{0})\leqslant c(\theta,\gamma_{1},\gamma_{2})[A(\rr_{1}-\rr_{0})^{-\gamma_{1}}+B(\rr_{1}-\rr_{0})^{-\gamma_{2}}].$
\end{lemma}
We will often consider the \ap quadratic'' version of the excess functional defined in~\eqref{eccess}, i.~\!e. 
\begin{equation}\label{eccquad}
\widetilde{\mathcal{E}}(w,z_0;B_\rr(x_0)) := \left( \ \dashint_{B_\rr(x_0)} \snr{V(\dd w)-z_0}^2 \dx \right)^\frac{1}{2}.
\end{equation}
In the particular case $z_0 = (\dd w)_{B_\rr(x_0)}$ ($z_0 =(V(\dd w))_{B_\rr(x_0)}$, resp.) we shall simply write $\mathcal{E}(w,(\dd w)_{B_\rr(x_0)}; B_\rr (x_0)) \equiv \mathcal{E}(w; B_\rr(x_0))$ ( $\widetilde{\mathcal{E}}(w,(V(\dd w))_{B_\rr(x_0)};B_\rr(x_0))\equiv \widetilde{\mathcal{E}}(w;B_\rr(x_0))$, resp.). A simple computation shows that
\begin{equation}\label{quad exc 1}
\mathcal{E}(w;B_\rr(x_0))^{p/2} \approx \widetilde{\mathcal{E}}(w;B_\rr(x_0)).
\end{equation}

Moreover, from \eqref{minav} and from \cite[Formula~(2.6)]{gm86} we have that
\begin{equation}\label{quad exc 2}
\widetilde{\mathcal{E}}(w;B_\rr(x_0)) \approx \widetilde{\mathcal{E}}(w, V((\dd w)_{B_\rr(x_0)});B_\rr(x_0)).
\end{equation}

\subsection{Basic regularity results} In this section we collect some basic estimates for local minimizers of nonhomogeneous quasiconvex functionals. We start with a variation of the classical Caccioppoli inequality accounting for the presence of a nontrivial right-hand side term, coupled with an higher integrability result of Gehring-type.
\begin{lemma}
	Under assumptions \eqref{assf}$_{1,2,3}$, \eqref{sqc} and \eqref{f}, let $u\in W^{1,p}(\Omega,\mathds{R}^{N})$ be a local minimizer of functional \eqref{main_fnc}.
	\begin{itemize}
		\item For every ball $B_{\rr}(x_{0})\Subset \Omega$ and any $u_{0}\in \mathds{R}^{N}$, $z_{0}\in \mathds{R}^{N\times n}\setminus \{0\}$ it holds that
		\begin{eqnarray}\label{cccp}
			\mathcal{E}(u,z_0;B_{\rr/2}(x_0))^p 
			&\leqslant& c\, \dashint_{B_{\rr}(x_{0})}\snr{z_{0}}^{p-2}\modulo{\frac{u-\ell}{\rr}}^{2}+\modulo{\frac{u-\ell}{\rr}}^{p} \dx \\*[0.5ex] &&+\,\frac{c}{\snr{z_{0}}^{p-2}}\left(\rr^{m}\dashint_{B_{\rr}(x_{0})}\snr{f}^{m} \  \dx\right)^{\frac{2}{m}},\notag
		\end{eqnarray}
		where~$\mathcal{E}(\cdot)$ is defined in \eqref{eccess}, $\ell(x):=u_{0}+\langle z_0,x-x_{0}\rangle$ and $c\equiv c(n,N,\lambda,\Lambda,p)$. 
		\item There exists an higher integrability exponent $p_{2}\equiv p_{2}(n,N,\lambda,\Lambda,p)>p$ such that $Du\in L^{p_{2}}_{\loc}(\Omega,\mathds{R}^{N\times n})$ and the reverse H\"older inequality
		\begin{eqnarray}\label{revh}
			&& \left(\ \dashint_{B_{\rr/2}(x_{0})}\snr{Du-(Du)_{B_{\rr}(x_{0})}}^{p_{2}} \dx\right)^{\frac{1}{p_{2}}}\notag\\
			&&\qquad \leqslant c\left( \ \dashint_{B_{\rr}(x_{0})}\snr{Du}^{p} \dx\right)^{\frac{1}{p}}+c\left(\rr^{m}\dashint_{B_{\rr}(x_{0})}\snr{f}^{m} \dx\right)^{\frac{1}{m(p-1)}}\,,
		\end{eqnarray}
		is verified for all balls $B_{\rr}(x_{0})\Subset \Omega$ with $c\equiv c(n,N,\lambda,\Lambda,p)$.
	\end{itemize}
\end{lemma}
\begin{proof}
	For the ease of exposition, we split the proof in two steps, each of them corresponding to the proof of \eqref{cccp} and \eqref{revh} respectively.
	\vspace{-4mm}
	\subsubsection*{Step 1: proof of \eqref{cccp} and \eqref{cccp 2}} We choose parameters $\rr/2\leqslant\tau_{1}<\tau_{2}\leqslant\rr$, a cut-off function $\eta\in C^{1}_{c}(B_{\tau_{2}}(x_{0}))$ such that $\mathds{1}_{B_{\tau_{1}}(x_{0})}\leqslant\eta\leqslant\mathds{1}_{B_{\tau_{2}}(x_{0})}$ and $\snr{D\eta}\lesssim (\tau_{2}-\tau_{1})^{-1}$. Set $\varphi_{1}:=\eta(u-\ell)$, $\varphi_{2}:=(1-\eta)(u-\ell)$ and use \eqref{sqc} and the equivalence in \eqref{Vm}$_{1}$ to estimate
	\begin{eqnarray}\label{0}
		c \int_{B_{\tau_{2}}(x_{0})} \snr{V_{\snr{z_{0}}}(D\varphi_{1})}^{2} \  \dx & \leqslant & \int_{B_{\tau_{2}}(x_{0})} [F(z_{0}+D\F_{1})-F(z_{0})] \dx \notag\\*[0.5ex]
		& = & \int_{B_{\tau_{2}}(x_{0})}[F(Du-D\F_{2})-F(Du)] \dx \notag\\*[0.5ex]
		&& +\int_{B_{\tau_{2}}(x_{0})}[F(Du)-F(Du-D\F_{1})] \dx\notag\\*[0.5ex]
		&& +\int_{B_{\tau_{2}}(x_{0})} [F(z_{0}+D\F_{2})-F(z_{0})] \dx=: \mbox{I}_1 +\mbox{I}_2+ \mbox{I}_3,
	\end{eqnarray}
	where we have used the simple relation~$D\F_{1} + D\F_{2} = Du -z_{0}$. Terms $\mbox{I}_{1}$ and $\mbox{I}_{3}$ can be controlled as done in \cite[Proposition 2]{dumi}; indeed we have
	\begin{eqnarray}\label{1}
		\mbox{I}_{1}+\mbox{I}_{3}&\leqslant& c\int_{B_{\tau_{2}}(x_{0})\setminus B_{\tau_{1}}(x_{0})}\snr{V_{\snr{z_{0}}}(D\varphi_{2})}^{2} \dx+c\int_{B_{\tau_{2}}(x_{0})\setminus B_{\tau_{1}}(x_{0})}\snr{V_{\snr{z_{0}}}(Du-z_{0})}^{2} \dx\nonumber \\*[0.5ex]
		&\stackrel{\eqref{Vm}_{2}}{\leqslant}&c\int_{B_{\tau_{2}}(x_{0})\setminus B_{\tau_{1}}(x_{0})}\snr{V_{\snr{z_{0}}}(Du-z_{0})}^{2}+\left|\ V_{\snr{z_{0}}}\left(\frac{u-\ell}{\tau_{2}-\tau_{1}}\right)\ \right|^{2} \dx,
	\end{eqnarray}
	for $c\equiv c(n,N,\lambda,\Lambda,p)$. Concerning term $\mbox{I}_{2}$, we exploit \eqref{f}, the fact that $\varphi_{1}\in W^{1,p}_{0}(B_{\tau_{2}}(x_{0}),\mathds{R}^{N})$ and apply Sobolev-Poincar\'e inequality to get
	\begin{eqnarray}\label{2}
		\mbox{I}_{2}&\leqslant&\snr{B_{\tau_{2}}(x_{0})}\left(\tau_{2}^{m}\dashint_{B_{\tau_{2}}(x_{0})}\snr{f}^{m}\  \dx\right)^{1/m}\left(\tau_{2}^{-m'}\dashint_{B_{\tau_{2}}(x_{0})}\snr{\varphi_{1}}^{m'}\  \dx\right)^{\frac{1}{m'}}\nonumber \\*[0.5ex]
		&\leqslant&\snr{B_{\tau_{2}}(x_{0})}\left(\tau_{2}^{m}\dashint_{B_{\tau_{2}}(x_{0})}\snr{f}^{m}\  \dx\right)^{1/m}\left( \ \dashint_{B_{\tau_{2}}(x_{0})}\left| \ \frac{\varphi_{1}}{\tau_{2}} \ \right|^{2^{*}}\  \dx\right)^{\frac{1}{2^{*}}}\nonumber \\*[0.5ex]
		&\leqslant&\snr{B_{\tau_{2}}(x_{0})}\left(\tau_{2}^{m}\dashint_{B_{\tau_{2}}(x_{0})}\snr{f}^{m} \dx\right)^{1/m}\left( \ \dashint_{B_{\tau_{2}}(x_{0})}\snr{D\varphi_{1}}^{2} \dx\right)^{\frac{1}{2}}\nonumber \\*[0.5ex]
		&\leqslant&\varepsilon\int_{B_{\tau_{2}}(x_{0})}\snr{V_{\snr{z_{0}}}(D\varphi_{1})}^{2} \dx+\frac{c\snr{B_{\rr}(x_{0})}}{\varepsilon\snr{z_{0}}^{p-2}}\left(\rr^{m}\dashint_{B_{\rr}(x_{0})}\snr{f}^{m} \dx\right)^{\frac{2}{m}},
	\end{eqnarray}
	where $c\equiv c(n,N,m)$ and we also used that $\rr/2\leqslant\tau_{2}\leqslant\rr$. Merging the content of the two above displays, recalling that $\eta\equiv 1$ on $B_{\tau_{1}}(x_{0})$
	and choosing $\varepsilon>0$ sufficiently small, we obtain
	\begin{eqnarray*}
		\int_{B_{\tau_{1}}(x_{0})}\snr{V_{\snr{z_{0}}}(Du-z_{0})}^{2} \dx&\leqslant&c\int_{B_{\tau_{2}}(x_{0})\setminus B_{\tau_{1}}(x_{0})}\snr{V_{\snr{z_{0}}}(Du-z_{0})}^{2}+\left| \ V_{\snr{z_{0}}}\left(\frac{u-\ell}{\tau_{2}-\tau_{1}}\right)\ \right|^{2} \dx \nonumber \\*[0.5ex]
		&&+\frac{c\snr{B_{\rr}(x_{0})}}{\snr{z_{0}}^{p-2}}\left(\rr^{m}\dashint_{B_{\rr}(x_{0})}\snr{f}^{m} \dx\right)^{\frac{2}{m}},
	\end{eqnarray*}
	with $c\equiv c(n,N,\lambda,\Lambda,p)$. At this stage, the classical hole-filling technique, Lemma \ref{l5} and \eqref{equiv.1} yield \eqref{cccp} and the first bound in the statement is proven.

	\subsubsection*{Step 2: proof of \eqref{revh}} To show the validity of \eqref{revh}, we follow \cite[proof of Proposition 3.2]{kumi} and first observe that if $u$ is a local minimizer of functional $\mathcal{F}(\cdot)$ on $B_{\rr}(x_{0})$, setting $f_{\rr}(x):=\rr f(x_{0}+\rr x)$, the map $u_{\rr}(x):=\rr^{-1}u(x_{0}+\rr x)$ is a local minimizer on $B_{1}(0)$ of an integral with the same integrand appearing in \eqref{main_fnc} satisfying $\eqref{assf}_{1,2,3}$ and $f_{\rr}$ replacing $f$. This means that \eqref{1} still holds for all balls $B_{\sigma/2}(\ti{x})\subseteq B_{\tau_{1}}(\ti{x})\subset B_{\tau_{2}}(\ti{x})\subseteq B_{\sigma}(\ti{x})\Subset B_{1}(0)$, with $\ti{x}\in B_{1}(0)$ being any point - in particular it remains true if $\snr{z_{0}}=0$, while condition $\snr{z_{0}}\not =0$ was needed only in the estimate of term $\mbox{I}_{2}$ in \eqref{2}, that now requires some change. So, in the definition of the affine map $\ell$ we choose $z_{0}=0$, $u_{0}=(u_{\rr})_{B_{\sigma}(\ti{x})}$ and rearrange estimates \eqref{1}-\eqref{2} as:
	\[
		\mbox{I}_{1}+\mbox{I}_{3}\stackrel{\eqref{equiv.1}}{\leqslant} c\int_{B_{\tau_{2}}(\ti{x})\setminus B_{\tau_{1}}(\ti{x})}\snr{Du_{\rr}}^{p} +\left| \ \frac{u_{\rr}-(u_{\rr})_{B_{\sigma}(\ti{x})}}{\tau_{2}-\tau_{1}} \ \right|^{p} \dx\,,
   \]
	and, recalling that $\varphi_{1}\in W^{1,p}_{0}(B_{\tau_{2}}(\ti{x}),\mathds{R}^{N})$, via Sobolev Poincaré, H\"older and Young inequalities and \eqref{f.0}$_{2}$, we estimate
	\begin{eqnarray*}
		\mbox{I}_2 &\leqslant &  \modulo{B_{\tau_2}(\ti{x})} \left(\tau_2^{(p^*)'} \dashint_{B_{\tau_2}(\ti{x})}\modulo{f_{\rr}}^{(p^*)'} \  \dx\right)^{\frac{1}{(p^*)'}} \left(\tau_2^{-p^*} \dashint_{B_{\tau_2}(\ti{x})}\modulo{\F_1}^{p^*}\  \dx \right)^\frac{1}{p^*}\\*[0.5ex]
		&\leqslant & c\modulo{B_{\tau_2}(\ti{x})} \left(\tau_2^{(p^*)'} \dashint_{B_{\tau_2}(\ti{x})}\modulo{f_{\rr}}^{(p^*)'} \dx\right)^{\frac{1}{(p^*)'}} \left( \ \dashint_{B_{\tau_2}(\ti{x})}\modulo{\dd \F_1}^{p} \dx \right)^\frac{1}{p}\\*[0.5ex]
		&\leqslant & \frac{c\snr{B_{\sigma}(\ti{x})}}{\varepsilon^{1/(p-1)}}\left(\sigma^{(p^*)'} \dashint_{B_\sigma(\ti{x})}\modulo{f_{\rr}}^{(p^*)'} \dx\right)^{\frac{p}{(p^*)'(p-1)}}+ \varepsilon  \int_{B_{\tau_2}(\ti{x})}\modulo{\dd \F_1}^{p} \dx,
	\end{eqnarray*}
	with $c\equiv c(n,N,p)$. Plugging the content of the two previous displays in \eqref{0}, reabsorbing terms and applying Lemma \ref{l5}, we obtain
	\begin{eqnarray}\label{3}
		&&\dashint_{B_{\sigma/2}(\ti{x})}\snr{Du_{\rr}}^{p} \ \dx\notag \\*[0.5ex]
		&&\quad\leqslant c \ \dashint_{B_{\sigma}(\ti{x})}\left| \ \frac{u_{\rr}-(u_{\rr})_{B_{\sigma}(\ti{x})}}{\sigma} \ \right|^{p} \ \dx+c\left(\sigma^{(p^*)'} \dashint_{B_\sigma(\ti{x})}\modulo{f_{\rr}}^{(p^*)'} \dx\right)^{\frac{p}{(p^*)'(p-1)}},
	\end{eqnarray}
	for $c\equiv c(n,N,\Lambda,\lambda,p)$. Notice that 
	\begin{flalign}\label{5}
		n\left(\frac{p}{(p^*)'(p-1)}-1\right)\leqslant\frac{p}{p-1},
	\end{flalign}
	with equality holding when $p<n$, while for $p\geqslant n$ any value of $p^{*}>1$ will do. 
	We then manipulate the second term on the right-hand side of \eqref{3} as
	\begin{eqnarray*}
	&&	\left(\sigma^{(p^*)'} \dashint_{B_\sigma(\ti{x})}\modulo{f_{\rr}}^{(p^*)'} \dx\right)^{\frac{p}{(p^*)'(p-1)}}\notag\\*[0.5ex]
	&&\quad\leqslant\sigma^{\frac{p}{p-1}-n\left(\frac{p}{(p^{*})'(p-1)}-1\right)}\left( \ \dashint_{B_{1}(0)}\snr{f_{\rr}}^{(p^{*})'} \ \dx\right)^{\frac{p}{(p^{*})'(p-1)}-1}		\dashint_{B_{\sigma}(\ti{x})}\snr{f_{\rr}}^{(p^{*})'} \dx\nonumber \\*[0.5ex]
		&&\quad\stackrel{\eqref{5}}{\leqslant}\left( \ \dashint_{B_{1}(0)}\snr{f_{\rr}}^{(p^{*})'} \ \dx\right)^{\frac{p}{(p^{*})'(p-1)}-1}\dashint_{B_{\sigma}(\ti{x})}\snr{f_{\rr}}^{(p^{*})'} \ \dx\nonumber \\*[0.5ex]
		&&\quad=:\dashint_{B_{\sigma}(\ti{x})}\snr{\mathfrak{K}_{\rr}f_{\rr}}^{(p^{*})'} \ \dx,
	\end{eqnarray*}
	where we set $\mathfrak{K}_{\rr}^{(p^{*})'}:=\snr{B_{1}(0)}^{1-\frac{p}{(p^{*})'(p-1)}}\nr{f_{\rr}}^{\frac{p}{p-1}-(p^{*})'}_{L^{(p^{*})'}(B_{1}(0))}$. Plugging the content of the previous display in \eqref{3} and applying Sobolev-Poincar\'e inequality we get
	\begin{eqnarray*}
		\dashint_{B_{\sigma/2}(\ti{x})}\snr{Du_{\rr}}^{p} \dx\leqslant c\left( \ \dashint_{B_{\sigma}(\ti{x})}\snr{Du_{\rr}}^{p_{*}} \ \dx\right)^{\frac{p}{p_{*}}}+c \ \dashint_{B_{\sigma}(\ti{x})}\snr{\mathfrak{K}_{\rr}f_{\rr}}^{(p^{*})'} \ \dx,
	\end{eqnarray*}
	with $c\equiv c(n,N,\Lambda,\lambda,p)$. Now we can apply a variant of Gehring  lemma \cite[Corollary 6.1]{giu} to determine a higher integrability exponent $\mathfrak{s}\equiv \mathfrak{s}(n,N,\Lambda,\lambda,p)$ such that $1<\mathfrak{s}\leqslant m/(p^{*})'$ and
	\begin{eqnarray}\label{4}
		\left( \ \dashint_{B_{\sigma/2}(\ti{x})}\snr{Du_{\rr}}^{\mathfrak{s}p} \dx\right)^{\frac{1}{\mathfrak{s}p}}&\leqslant& c\left( \ \dashint_{B_{\sigma}(\ti{x})}\snr{Du_{\rr}}^{p}\dx\right)^{\frac{1}{p}}\\*[0.5ex]
		&&+c\mathfrak{K}_{\rr}^{(p^{*})'/p}\left( \ \dashint_{B_{\sigma}(\ti{x})}\snr{f_{\rr}}^{\mathfrak{s}(p^{*})'} \ \dx\right)^{\frac{1}{\mathfrak{s}p}}
	\end{eqnarray}
	for $c\equiv c(n,N,\Lambda,\lambda,p)$. Next, notice that
	\begin{flalign*}
		\mathfrak{K}_{\rr}^{(p^{*})'/p}=\left( \ \dashint_{B_{1}(0)}\snr{f_{\rr}}^{(p^{*})'} \dx\right)^{\frac{1}{(p^{*})'(p-1)}-\frac{1}{p}}\leqslant\left( \ \dashint_{B_{1}(0)}\snr{f_{\rr}}^{\mathfrak{s}(p^{*})'} \dx\right)^{\frac{1}{\mathfrak{s}(p^{*})'(p-1)}-\frac{1}{\mathfrak{s}p}},
	\end{flalign*}
	so plugging this last inequality in \eqref{4} and recalling that $\mathfrak{s}(p^{*})'\leqslant m$, we obtain
	\begin{flalign*}
		\left( \ \dashint_{B_{\sigma/2}(\ti{x})}\snr{Du_{\rr}}^{\mathfrak{s}p} \dx\right)^{\frac{1}{\mathfrak{s}p}}\leqslant c\left( \ \dashint_{B_{\sigma}(\ti{x})}\snr{Du_{\rr}}^{p}\dx\right)^{\frac{1}{p}}+c\left( \ \dashint_{B_{\sigma}(\ti{x})}\snr{f_{\rr}}^{m} \ \dx\right)^{\frac{1}{m(p-1)}}.
	\end{flalign*}
	Setting $p_{2}:=\mathfrak{s}p>p$ above and recalling that $\ti{x}\in B_{1}(0)$ is arbitrary, we can fix $\ti{x}=0$, scale back to $B_{\rr}(x_{0})$ and apply \eqref{minav} to get \eqref{revh} and the proof is complete.
\end{proof}

\section{Excess decay estimate}\label{excess_decay}
In this section we prove some excess decay estimates considering separately two cases: when a smallness condition on the excess functional of our local minimizer $u$ is satisfied and when such an estimate does not hold true.

\subsection{The nondegenerate scenario}
We start working assuming that a suitable smallness condition on the excess functional $\mathcal{E}(u;B_\rr(x_0))$ is fulfilled. In particular, we prove the following proposition.
\begin{proposition}\label{l2}
	Under assumptions \eqref{assf}$_{1,2,3}$, \eqref{sqc} and \eqref{f}, let $u\in W^{1,p}(\Omega,\mathds{R}^{N})$ be a local minimizer of functional \eqref{main_fnc}. Then, for $\tau_0 \in (0,2^{-10})$, there exists $\eps_0 \equiv \eps_0(\data,\tau_0) \in (0,1)$ and $\eps_1 \equiv \eps_1(\data,\tau_0) \in (0,1)$ such that the following implications hold true.
	\begin{itemize}
		\item If the conditions
		\begin{equation}\label{l2 2}
		\mathcal{E}(u;B_\rr(x_0)) \leqslant\eps_0 \snr{(\dd u )_{B_\rr (x_0)}},
		\end{equation}
		and
		\begin{equation}\label{l2 3}
		\left(\rr^m \dashint_{B_{\rr}(x_{0})}\snr{f}^m \ \dx \right)^\frac{1}{m} \leqslant \eps_1 \snr{(\dd u )_{B_\rr (x_0)}}^\frac{p-2}{2} \mathcal{E}(u; B_\rr(x_0))^\frac{p}{2},
		\end{equation}
		are verified on $B_\rr(x_0)$, then it holds that
		\begin{equation}\label{l2 4}
		\mathcal{E}(u; B_{\tau_0 \rr}(x_0)) \leqslant c_0 \tau_0^{\beta_0} \mathcal{E}(u;B_\rr(x_0)),
		\end{equation}
		for all $\beta_0 \in (0,2/p)$, with $c_0 \equiv c_0(\data)>0$.
		\item If condition \eqref{l2 2} holds true and
		\begin{equation}\label{l3 1}
		\left(\rr^m \dashint_{B_{\rr}(x_{0})}\snr{f}^m \ \dx \right)^\frac{1}{m} >  \eps_1 \snr{(\dd u )_{B_\rr (x_0)}}^\frac{p-2}{2} \mathcal{E}(u; B_\rr(x_0))^\frac{p}{2},
		\end{equation}
		is satisfied on $B_\rr(x_0)$, then
		\begin{equation}\label{l3 2}
		\mathcal{E}(u; B_{\tau_0 \rr}(x_0)) \leqslant c_0 \left(\rr^m \dashint_{B_{\rr}(x_{0})}\snr{f}^m \ \dx \right)^\frac{1}{m(p-1)},
		\end{equation}
		for $c_0 \equiv c_0(\data)>0$.
	\end{itemize}
\end{proposition}

\begin{proof}[Proof of Proposition {\rm\ref{l2}}]
	For the sake of readability, since all balls considered here are concentric to $B_\rr(x_0)$, we will omit denoting the center. Moreover, we will adopt the following notation $(\dd u )_{B_\sss (x_0)} \equiv (\dd u )_\sss$ and, for all $\F \in C^\infty_c(B_\rr;\erN)$, we will denote $\norma{\dd \F}_{L^\infty(B_\rr)} \equiv \norma{\dd \F}_\infty$. We spilt the proof in two steps.
	
	\subsubsection*{Step 1: proof of \eqref{l2 4}} With no loss of generality we can assume that $\mathcal{E}(u;B_\rr)> 0$, which clearly implies, thanks to \eqref{l2 2}, that $\snr{(\dd u )_\rr}>0$.
	
	We begin proving that condition \eqref{l2 2} implies that
	\begin{equation} \label{l2 5}
	\dashint_{B_\rr} \snr{\dd u}^p \dx \leqslant c \snr{(\dd u)_\rr}^p,
	\end{equation}
	for a constant $c \equiv c(p,\eps_0)>0$. Indeed,
	\begin{eqnarray*}
		\dashint_{B_\rr} \snr{\dd u}^p \dx &\leqslant& c \ \dashint_{B_\rr} \snr{\dd u - (\dd u)_\rr}^p \dx + c \snr{(\dd u)_\rr}^p\\*[0.5ex]
		&\stackrel{\eqref{eccess}}{\leqslant}& c \, \mathcal{E}(u;B_\rr)^p + c\snr{(\dd u)_\rr}^p\\*[0.5ex]
		&\stackrel{\eqref{l2 2}}{\leqslant}& c(\eps_0^p+1) \snr{(\dd u)_\rr}^p,
	\end{eqnarray*}
	and \eqref{l2 5} follows.
	
	Consider now 
	\begin{equation}\label{u0}
	B_\rr \ni x \mapsto u_0(x) := \frac{\snr{(\dd u )_\rr}^\frac{p-2}{2}\big(u(x)-(u )_\rr - \langle (\dd u )_\rr,x-x_0\rangle\big)}{\mathcal{E}(u; B_\rr)^{p/2}},
	\end{equation}
	and 
	$$
    d:= \left(\frac{\mathcal{E}(u;B_\rr)}{\snr{(\dd u)_\rr}}\right)^\frac{p}{2}.
	$$
	Let us note that we have
	\begin{eqnarray*}
	&& \dashint_{B_\rr}\snr{\dd u_0}^2 \dx + d^{p-2}\dashint_{B_\rr}\snr{\dd u_0}^2 \dx \\*[0.5ex]
    && \qquad \leqslant   \frac{\snr{(\dd u)_\rr}^{p-2}}{\mathcal{E}(u;B_\rr)^p}\dashint_{B_\rr}\snr{\dd u-(\dd u)_\rr}^2 \dx \\*[0.5ex]
    &&\qquad\qquad +  \left(\frac{\mathcal{E}(u;B_\rr)}{\snr{(\dd u)_\rr}}\right)^\frac{p(p-2)}{2}\frac{\snr{(\dd u)_\rr}^\frac{p(p-2)}{2}}{\mathcal{E}(u;B_\rr)^\frac{p^2}{2}}\dashint_{B_\rr}\snr{\dd u-(\dd u)_\rr}^p  \dx\\*[0.5ex]
     && \qquad \leqslant   \frac{1}{\mathcal{E}(u;B_\rr)^p}\dashint_{B_\rr}\snr{(\dd u)_\rr}^{p-2}\snr{\dd u-(\dd u)_\rr}^2 \dx \\*[0.5ex]
     &&\qquad\quad +  \frac{1}{\mathcal{E}(u;B_\rr)^p}\dashint_{B_\rr}\snr{\dd u-(\dd u)_\rr}^p  \dx \leqslant 1.
	\end{eqnarray*}	
	Since $\snr{(\dd u )_\rr} >0$ we have that the hypothesis of \cite[Lemma~3.2]{dqc} are satisfied with
	\begin{equation}\label{A}
	\mathscr{A} := \partial^2 F((\dd u )_\rr)\snr{(\dd u )_\rr}^{2-p}.
	\end{equation}
	Then,	
	\begin{eqnarray*}
	\modulo{ \ \dashint_{B_\rr} \mathscr{A} \langle \dd u_0, \dd \F \rangle \ \dx} &\leqslant& \frac{c \norma{\dd \F}_\infty \snr{(\dd u )_\rr}^\frac{2-p}{2}}{\mathcal{E}(u;B_\rr)^\frac{p}{2}} \left( \rr^m \dashint_{B_\rr} \snr{f}^m \ \dx\right)^\frac{1}{m}\\*[0.5ex]
		&& + c\norma{\dd \F}_\infty \mu \left(\frac{\mathcal{E}(u;B_\rr)}{\snr{(\dd u)_\rr}} \right)^\frac{1}{p} \left[1 + \left(\frac{\mathcal{E}(u;B_\rr)}{\snr{(\dd u)_\rr}}\right)^\frac{p-2}{2} \right]\\*[0.5ex]
		&\stackrel{\eqref{l2 2},\eqref{l2 3}}{\leqslant}& c \eps_1 \norma{\dd \F}_\infty +c \norma{\dd \F}_\infty \mu(\eps_0)^\frac{1}{p}\big[1+ \eps_0^\frac{p-2}{2} \big].
	\end{eqnarray*}
	Fix $\eps >0$ and let $\delta \equiv \delta(\data,\eps)>0$ be the one given by \cite[Lemma~2.4]{kumi} and choose $\eps_0$ and $\eps_1$ sufficiently small such that
	\begin{equation}\label{choice_eps}
	c \, \eps_1 + c \mu(\eps_0)^\frac{1}{p}\big[1+ \eps_0^\frac{p-2}{2} \big] \leqslant \delta   .
    \end{equation}
	With this choice of $\eps_0$ and $\eps_1$ it follows that $u_0$ is almost $\mathscr{A}$-harmonic on $B_\rr$, in the sense that
	$$
	\modulo{ \ \dashint_{B_\rr}\mathscr{A}\langle \dd u_0, \dd \F \rangle \ \dx} \leqslant\delta \norma{\dd \F}_\infty,
	$$
	with $\mathscr{A}$ as in \eqref{A}. Hence, by \cite[Lemma~2.4]{kumi}  we obtain that there exists $h_0 \in W^{1,2}(B_\rr;\erN)$ which is $\mathscr{A}$-harmonic, i.e.
	$$
	\int_{B_\rr}\mathscr{A}\langle \dd h_0, \dd \F \rangle \ \dx =0 \qquad \mbox{for all } \F \in C^\infty_c(B_\rr;\erN),
	$$
	such that
	\begin{equation} \label{l2 6}
     \dashint_{B_{3\rr/4}}\snr{\dd h_0}^2 \dx + d^{p-2}\dashint_{B_{3\rr/4}}\snr{\dd h_0}^p \dx \leqslant 8^{2np}\,,
	\end{equation}
    and 
    	\begin{equation} \label{l2 6.2}
    	\dashint_{B_{3\rr/4}}\modulo{\frac{u_0-h_0}{\rr}}^2 + d^{p-2}\modulo{\frac{u_0-h_0}{\rr}}^p \dx \leqslant \eps.
    \end{equation}
    We choose now~$\tau_0\in (0,2^{-10})$, which will be fixed later on, and estimate
    \begin{eqnarray}\label{l2 7}
	&& \dashint_{B_{2\tau_0\rr}}\modulo{\frac{u_0(x)-h_0(x_{0})-\langle Dh_0(x_{0}),x-x_0\rangle}{\tau_0\rr}}^2 \dx\nonumber \\*[0.5ex]
	&& \quad  \leqslant c \,\dashint_{B_{2\tau_0\rr}}\modulo{\frac{h_0(x)-h_0(x_0)-\langle Dh_0(x_0),x-x_0\rangle}{\tau_0\rr}}^2 \dx+c\,\dashint_{B_{2\tau_0\rr}}\modulo{\frac{u_0-h_0}{\tau_0\rr}}^2 \dx \nonumber \\*[0.5ex]
	&& \quad \stackrel{\eqref{l2 6.2}}{\leqslant} c(\tau_0\rr)^2\sup_{B_{\rr/2}}\snr{\dd^2 h_0}^2 +\frac{c\varepsilon}{\tau_0^{n+2}}\nonumber \\*[0.5ex]
	&&\quad \leqslant c\,\tau_0^{2}\dashint_{B_{3\rr/4}}\snr{Dh_0}^{2} \dx +\frac{c\varepsilon}{\tau_0^{n+2}}\notag\\*[0.5ex] 
	&&\quad \stackrel{\eqref{l2 6}}{\leqslant} c\,\tau_0^{2}+\frac{c\varepsilon}{\tau_0^{n+2}},
    \end{eqnarray}
    where~$c\equiv c(\data)>0$ and where we have used the following property of ~$\mathscr{A}$-harmonic functions 
   \begin{equation}\label{property_A_arm}
    \rr^\gamma \sup_{B_{\rr/2}}\snr{\dd^2 h_0}^\gamma \leqslant c \, \dashint_{B_{3\rr/4}}\snr{\dd h_0}^\gamma \dx\,,
   \end{equation}
     with~$\gamma>1$ and~$c$ depending on~$n$, $N$, and on the ellipticity constants of~$\mathscr{A}$.
     
     Now, choosing
    \[
    \eps := \tau_0^{n+2p}\,,
    \]
    we have that this together with~\eqref{choice_eps} gives that~$\eps_0 \equiv \eps_0(\data,\tau_0)$ and~$\eps_1\equiv \eps_1(\data,\tau_0)$. Recalling the definition of~$u_0$ in~\eqref{u0} and~\eqref{l2 7} we eventually arrive at
    \begin{eqnarray}\label{l2 8}
	&& \dashint_{B_{2\tau_0\rr}}\frac{\snr{u-(u)_{\rr}-\langle(\dd u)_{\rr},x-x_0\rangle-\snr{(\dd u)_\rr}^\frac{2-p}{2}\mathcal{E}(u;B_\rr)^{p/2}\left(h_0(x_0)-\langle  \dd h_0(x_0),x-x_0\rangle\right)}^2}{(\tau_0\rr)^2}\dx\notag\\*[0.5ex]
	&& \hspace{5cm} \leqslant c\,\snr{(\dd u)_\rr}^{2-p}\mathcal{E}(u;B_\rr)^p\tau_0^{2}\,,
    \end{eqnarray}
   for~$c\equiv c(\textnormal{\texttt{data}})>0$.  By a similar computation, always using~\eqref{property_A_arm},~\eqref{l2 6} and~\eqref{l2 6.2}, we obtain that
    \begin{eqnarray*}
	&& d^{p-2}\dashint_{B_{2\tau_0\rr}}\left|\frac{u_0-h_0(x_0)-\langle \dd h_0(x_0),x-x_0\rangle}{\tau_0\rr}\right|^{p} \dx\nonumber \\*[0.5ex]
	&&\qquad  \leqslant cd^{p-2}(\tau_0\rr)^{p}\sup_{B_{\rr/2}}\snr{\dd^{2}h_0}^{p}+\frac{c\,\varepsilon}{\tau_0^{n+p}}\leqslant c\,\tau_0^{p}.
    \end{eqnarray*}
  In this way, as for~\eqref{l2 8}, by the definition of~$u_0$ in~\eqref{u0}, we eventually arrive at
   \begin{eqnarray}\label{l2 9}
	&&\dashint_{B_{2\tau_0\rr}}\frac{\snr{u-(u)_{\rr}-\langle(\dd u)_{\rr},x-x_0\rangle-\snr{(\dd u)_\rr}^\frac{2-p}{2}\mathcal{E}(u;B_\rr)^{p/2}\left(h_0(x_0)-\langle  \dd h_0(x_0),x-x_0\rangle\right)}^p}{(\tau_0\rr)^p}\dx\notag\\*[0.5ex]
	&&\hspace{5cm} \leqslant c \, d^{2-p}\snr{(\dd u)_\rr}^\frac{p(2-p)}{2}\mathcal{E}(u;B_\rr)^\frac{p^2}{2} \tau_0^p\notag\\*[0.5ex]
	&& \hspace{5cm} \leqslant c \, \mathcal{E}(u;B_\rr)^p \tau_0^2\,,
    \end{eqnarray}
    with~$c\equiv c(\textnormal{\texttt{data}})$.
    
    Denote now with~$\ell_{2\tau_0\rr}$ the unique affine function such that
     \[
	\ell_{2\tau_0\rr}\mapsto \min_{\ell \ \text{affine}}\dashint_{B_{2\tau_0\rr}}\snr{u-\ell}^{2} \dx.
    \]
    Hence, by~\eqref{l2 8} and~\eqref{l2 9}, we conclude that
    \begin{equation}\label{l2 10}
    \dashint_{B_{2\tau_0\rr}}\snr{(\dd u)_\rr}^{p-2}\modulo{\frac{u-\ell_{2\tau_0\rr}}{2\tau_0\rr}} + \modulo{\frac{u-\ell_{2\tau_0\rr}}{2\tau_0\rr}}^p \dx \leqslant c\, \tau^2 \mathcal{E}(u;B_\rr)^p.
    \end{equation}
    Notice that we have also used the property that
    \[
    \dashint_{B_\rr}\snr{u-\ell_{\rr}}^p \dx \leqslant c\,\dashint_{B_\rr}\snr{u -\ell}^p\dx\,,
    \]
   for~$p\geqslant 2$,~$c \equiv c(n,N,p)>0$ and for any affine function~$\ell$; see~\cite[Lemma~2.3]{kumi}.
   
  Recalling the definition of the excess functional~$\mathcal{E}(\cdot)$, in~\eqref{eccess},  we can estimate the following quantity as follows
\begin{eqnarray}\label{l2 11}
	\snr{D\ell_{2\tau_0\rr}-(Du)_{\rr}} &\leqslant& \snr{D\ell_{2\tau_0\rr}-(Du)_{2\tau_0\rr}}+\snr{(Du)_{2\tau_0\rr}-(Du)_{\rr}}\nonumber \\*[0.5ex]
	&\leqslant& c\,\left(\dashint_{B_{2\tau_0\rr}}\snr{Du-(Du)_{2\tau_0\rr}}^{2} \dx\right)^{\frac{1}{2}}+\left(\dashint_{B_{2\tau_0\rr}}\snr{Du-(Du)_{\rr}}^{2} \dx\right)^{\frac{1}{2}}\nonumber \\*[0.5ex]
	&\stackrel{\eqref{minav}}{\leqslant}& \frac{c}{\tau_0^{n/2}}\left(\dashint_{B_{\rr}}\snr{Du-(Du)_{\rr}}^{2} \dx\right)^{\frac{1}{2}}\nonumber \\*[0.5ex]
	&=&\frac{c\snr{(Du)_{\rr}}^{\frac{2-p}{2}}}{\tau_0^{n/2}}\left(\dashint_{B_{\rr}}\snr{(Du)_{\rr}}^{p-2}\snr{Du-(Du)_{\rr}}^{2} \dx\right)^{\frac{1}{2}}\nonumber \\
	&\leqslant&\frac{c(n)}{\tau_0^{n/2}}\left(\frac{\mathcal{E}(u,B_\rr)}{\snr{(Du)_\rr}}\right)^{\frac{p}{2}}\snr{(Du)_{\rr}},
\end{eqnarray}
where we have used the following property of the affine function~$\ell_{2\tau_0\rr}$
\[
\snr{\dd \ell_{2\tau_0\rr} -(\dd u)_{2\tau_0\rr}}^p \leqslant c \, \dashint_{B_{2\tau_0\rr}}\snr{\dd u -(\dd u)_{2\tau_0\rr}}^p \dx \,,
\]
for a constant~$c \equiv c(n,p)>0$; see for example~\cite[Lemma~2.2]{kumi}.

Now, starting from~\eqref{l2 2} and~\eqref{choice_eps}, we further reduce the size of~$\eps_0$ such that
\begin{equation}\label{l2 12}
	\left(\frac{\mathcal{E}(u,B_{\rr})}{\snr{(Du)_{\rr}}}\right)^{\frac{p}{2}}\stackrel{\eqref{l2 2}}{\leqslant} \eps_0^{\frac{p}{2}} \leqslant\frac{\tau_0^{n/2}}{8c(n)},
\end{equation}
where $c\equiv c(n)$ is the same constant appearing in \eqref{l2 11}. Thus, combining~\eqref{l2 11} and~\eqref{l2 12}, we get
\begin{eqnarray}\label{l2 13}
	\snr{D\ell_{2\tau_0\rr}-(Du)_\rr}\leqslant\frac{\snr{(Du)_\rr}}{8}.
\end{eqnarray}
The information provided by~\eqref{l2 12} combined with~\eqref{l2 10} allow us to conclude that
\begin{equation}\label{l2 14}
\dashint_{B_{2\tau_0\rr}}\snr{D\ell_{2\tau_0\rr}}^{p-2}\modulo{\frac{u-\ell_{2\tau_0\rr}}{2\tau_0\rr}} + \modulo{\frac{u-\ell_{2\tau_0\rr}}{2\tau_0\rr}}^p \dx \leqslant c\, \tau^2 \mathcal{E}(u;B_\rr)^p.
\end{equation}
By triangular inequality and~\eqref{l2 13} we also get
\[
	\snr{D\ell_{2\tau_0\rr}}\geqslant \snr{(Du)_{\rr}}-\snr{D\ell_{2\tau_0\rr}-(Du)_{\rr}} \stackrel{\eqref{l2 13}}{\geqslant} \frac{7\snr{(Du)_{\rr}}}{8}
\]
which, therefore, implies that
\begin{eqnarray}\label{l2 15}
	&& \dashint_{B_{\tau_0\rr}}\snr{D\ell_{2\tau_0\rr}}^{p-2}\snr{Du-D\ell_{2\tau_0\rr}}^{2} \dx+\inf_{z\in \ernN}\dashint_{B_{\tau_0\rr}}\snr{Du-z}^p \dx\nonumber \\*[0.5ex]
	&&\qquad  \stackrel{\eqref{cccp}}{\leqslant} c\,\dashint_{B_{2\tau_0\rr}}\snr{D\ell_{2\tau_0\rr}}^{p-2}\left|\frac{u-\ell_{2\tau_0\rr}}{2\tau_0\rr}\right|^{2}+\left|\frac{u-\ell_{2\tau_0\rr}}{2\tau_0\rr}\right|^{p} \dx\nonumber \\*[0.5ex]
	&&\qquad \qquad   +\frac{c}{\snr{D\ell_{2\tau_0\rr}}^{p-2}}\left((2\tau_0\rr)^{m}\dashint_{B_{2\tau_0\rr}}\snr{f}^{m} \dx\right)^{\frac{2}{m}}\nonumber \\*[0.5ex]
	&&\qquad  \stackrel{\eqref{l2 14}}{\leqslant} c\,\tau_0^{2}\mathcal{E}(u,B_\rr)^{p}+\frac{c\tau_0^{2-2n/m}}{\snr{(Du)_\rr}^{p-2}}\left(\rr^{m}\dashint_{B_\rr}\snr{f}^{m} \dx\right)^{\frac{2}{m}}\,,
\end{eqnarray}
where~$c\equiv c(\data)>0$.
By triangular inequality, we can further estimate
\begin{eqnarray*}
	&&\dashint_{B_{\tau_0\rr}}\snr{(Du)_{\tau_0\rr}}^{p-2}\snr{Du-(Du)_{\tau_0\rr}}^{2} \dx \\*[0.5ex] && \qquad \leqslant
	c\,\dashint_{B_{\tau_0\rr}}\snr{D\ell_{\tau_0\rr}-(Du)_{\tau_0\rr}}^{p-2}\snr{Du-(Du)_{\tau_0\rr}}^{2} \dx\nonumber \\*[0.5ex]
	&&\qquad\quad +c\,\dashint_{B_{\tau_0\rr}}\snr{D\ell_{2\tau_0\rr}-D\ell_{\tau_0\rr}}^{p-2}\snr{Du-(Du)_{\tau_0\rr}}^{2} \dx\nonumber\\*[0.5ex]
	&&\qquad\quad +c\,\dashint_{B_{\tau_0\rr}}\snr{D\ell_{2\tau_0\rr}}^{p-2}\snr{Du-(Du)_{\tau_0\rr}}^{2} \dx\nonumber \\*[0.5ex]
	&&\qquad= \mbox{I}_1+\mbox{I}_2+\mbox{I}_3\,,
\end{eqnarray*}
where~$c\equiv c(p)>0$. We now separately estimate the previous integrals. We begin considering~$\mbox{I}_1$. By Young and triangular inequalities we get
\begin{eqnarray*}
	\mbox{I}_1 &\leqslant& c\snr{D\ell_{\tau_0\rr}-(Du)_{\tau_0\rr}}^{p}+c\,\dashint_{B_{\tau_0\rr}}\snr{Du-(Du)_{\tau_0\rr}}^{p} \dx\nonumber \\*[0.5ex]
	&\leqslant& c\,\dashint_{B_{\tau_0\rr}}\snr{Du-(Du)_{\tau_0\rr}}^{p} \dx\\*[0.5ex]
	&\stackrel{\eqref{minav}}{\leqslant}& c\inf_{z\in \erN}\dashint_{B_{\tau_0\rr}}\snr{Du-z}^{p} \dx\nonumber \\*[0.5ex]
	&\stackrel{\eqref{l2 15}}{\leqslant}& c\,\tau_0^{2}\mathcal{E}(u,B_\rr)^{p}+\frac{c\tau_0^{2-2n/m}}{\snr{(Du)_\rr}^{p-2}}\left(\rr^{m}\dashint_{B_\rr}\snr{f}^{m} \dx\right)^{\frac{2}{m}}\,,
\end{eqnarray*}
with $c\equiv c(\data)>0$. In a similar fashion, we can treat the integral~$\mbox{I}_2$ 
\begin{eqnarray*}
	\mbox{I}_2 &\leqslant& c\,\snr{D\ell_{2\tau_0\rr}-D\ell_{\tau_0\rr}}^{p}+c\,\dashint_{B_{\tau_0\rr}}\snr{Du-(Du)_{\tau_0\rr}}^{p} \dx\nonumber \\*[0.5ex]
	&\stackrel{\eqref{minav}}{\leqslant}& c\,\dashint_{B_{2\tau_0\rr}}\left|\frac{u-\ell_{2\tau_0\rr}}{2\tau_0\rr}\right|^{p} \dx+c\inf_{z\in \ernN}\dashint_{B_{\tau_0\rr}}\snr{Du-z}^{p} \dx\nonumber \\*[0.5ex]
	&\stackrel{\eqref{l2 14},\eqref{l2 15}}{\leqslant}& c\,\tau_0^{2}\mathcal{E}(u,B_\rr)^{p}+\frac{c\tau_0^{2-2n/m}}{\snr{(Du)_\rr}^{p-2}}\left(\rr^{m}\dashint_{B_\rr}\snr{f}^{m} \dx\right)^{\frac{2}{m}}\,,
\end{eqnarray*}
where we have used the following property of the affine function~$\ell_{2\tau_0\rr}$
\[
\snr{\dd \ell_{2\tau_0\rr} - \dd \ell_{\tau_0\rr} }^p \leqslant c\,\dashint_{B_{2\tau_0\rr}}\left|\frac{u-\ell_{2\tau_0\rr}}{2\tau_0\rr}\right|^{p} \dx\,,
\]
for a given constant~$c \equiv c(n,p)>0$; see~\cite[Lemma~2.2]{kumi}. Finally, the last integral~$\mbox{I}_3$ can be treated recalling~\eqref{l2 15} and~\eqref{minav}, i.~\!e.
\[
\mbox{I}_3 \leqslant c\,\tau_0^{2}\mathcal{E}(u,B_\rr)^{p}+\frac{c\tau_0^{2-2n/m}}{\snr{(Du)_\rr}^{p-2}}\left(\rr^{m}\dashint_{B_\rr}\snr{f}^{m} \dx\right)^{\frac{2}{m}}.
\]
All in all, combining the previous estimate
\begin{eqnarray*}
\mathcal{E}(u;B_{\tau_0\rr}) &\leqslant& c\,\tau_0^{2/p}\mathcal{E}(u,B_\rr)+\frac{c\tau_0^{2/p-2n/(mp)}}{\snr{(Du)_\rr}^\frac{p-2}{p}}\left(\rr^{m}\dashint_{B_\rr}\snr{f}^{m} \dx\right)^{\frac{2}{mp}}\\*[0.5ex]
&\stackrel{\eqref{l2 3}}{\leqslant}& c\,\tau_0^{2/p}\mathcal{E}(u,B_\rr)+c\tau_0^{2/p-2n/(mp)}\eps_1^{2/p}\mathcal{E}(u;B_{\tau_0\rr})\\*[0.5ex]
&\leqslant& c_0\tau_0^{2/p}\mathcal{E}(u;B_{\tau_0\rr})\,,
\end{eqnarray*}
up to choosing~$\eps_1$ such that
\[
\eps_1 \leqslant\tau_0^{n/m}.
\]
	\subsubsection*{Step 2: proof of \eqref{l3 2}} 	   
	The proof follows by \cite[Lemma 2.4]{dqc} which yields
	\begin{eqnarray*}
		\mathcal{E}(u; B_{\tau_0 \rr}(x_0))^\frac{p}{2} &\leqslant& \frac{2^{3p}}{\tau_0^{n/2}} \mathcal{E}(u; B_{ \rr}(x_0))^\frac{p}{2}\\*[0.5ex]
		&\stackrel{\eqref{l3 1}}{\leqslant}&  \frac{2^{3p}}{\tau_0^{n/2}} \eps_1^{-1} \snr{(\dd u )_{B_\rr (x_0)}}^\frac{2-p}{2}	\left(\rr^m \dashint_{B_{\rr}(x_{0})}\snr{f}^m \ \dx \right)^\frac{1}{m} \\*[0.5ex]
		&\stackrel{\eqref{l2 2}}{\leqslant}& \frac{2^{6(p-1)}}{\tau_0^{n(p-1)/p}} \eps_0^\frac{p-2}{2} \eps_1^{-1} \mathcal{E}(u;B_{\tau_0\rr})^\frac{2-p}{2}	\left(\rr^m \dashint_{B_{\rr}(x_{0})}\snr{f}^m \ \dx \right)^\frac{1}{m}.
	\end{eqnarray*}
	Multiplying both sides by $\mathcal{E}(u;B_{\tau_0\rr})^\frac{p-2}{2}$ we get the desired estimate.
\end{proof}

\subsection{The degenerate scenario}
It remains to considering the case when condition \eqref{l2 2} does not hold true. We start with two technical lemmas. The first one is an analogous of the Caccioppoli inequality \eqref{cccp}, where we take in consideration the eventuality $z_0 =0$.

\begin{lemma}
	Under assumptions \eqref{assf}$_{1,2,3}$, \eqref{sqc} and \eqref{f}, let $u\in W^{1,p}(\Omega,\mathds{R}^{N})$ be a local minimizer of functional \eqref{main_fnc}.
	For every ball $B_{\rr}(x_{0})\Subset \Omega$ and any $u_{0}\in \mathds{R}^{N}$, $z_{0}\in \mathds{R}^{N\times n}$ it holds that
	\begin{eqnarray} \label{cccp 2}
		 		\mathcal{E}(u,z_0;B_{\rr/2}(x_0))^p &\leqslant& c \ \dashint_{B_\rr(x_0)} \snr{z_0}^{p-2}\modulo{\frac{u-\ell}{\rr}}^2
		 		+      \modulo{\frac{u-\ell}{\rr}}^p \ \dx \\*[0.5ex]  
		 		&&+\, c \left(\rr^{m} \dashint_{B_\rr(x_0)} \snr{f}^m \ \dx \right)^\frac{p}{m(p-1)}\,,\notag
	\end{eqnarray}
	where~$\mathcal{E}(\cdot)$ is defined in~\eqref{eccess}, $\ell(x):=u_{0}+\langle z_0,x-x_{0}\rangle$ and $c\equiv c(n,N,\lambda,\Lambda,p)$. 
\end{lemma}
\begin{proof}
	The proof is analogous to estimate \eqref{cccp}, up to treating in a different way the term $\mbox{I}_2$ in \eqref{0}, taking in consideration the eventuality $z_0=0$. Exploiting \eqref{f} and fact that $\varphi_{1}\in W^{1,p}_{0}(B_{\tau_{2}}(x_{0}),\mathds{R}^{N})$, an application of the Sobolev-Poincar\'e inequality yields
	\begin{eqnarray}
		\mbox{I}_{2}&\leqslant&\snr{B_{\tau_{2}}(x_{0})}\left(\tau_{2}^{m}\dashint_{B_{\tau_{2}}(x_{0})}\snr{f}^{m}\  \dx\right)^{1/m}\left(\tau_{2}^{-m'}\dashint_{B_{\tau_{2}}(x_{0})}\snr{\varphi_{1}}^{m'}\  \dx\right)^{\frac{1}{m'}}\nonumber \\*[0.5ex]
		&\leqslant&\snr{B_{\tau_{2}}(x_{0})}\left(\tau_{2}^{m}\dashint_{B_{\tau_{2}}(x_{0})}\snr{f}^{m}\  \dx\right)^{1/m}\left( \ \dashint_{B_{\tau_{2}}(x_{0})}\left| \ \frac{\varphi_{1}}{\tau_{2}} \ \right|^{p^*}\  \dx\right)^{\frac{1}{p^*}}\nonumber \\*[0.5ex]
		&\leqslant&\snr{B_{\tau_{2}}(x_{0})}\left(\tau_{2}^{m}\dashint_{B_{\tau_{2}}(x_{0})}\snr{f}^{m} \dx\right)^{1/m}\left( \ \dashint_{B_{\tau_{2}}(x_{0})}\snr{D\varphi_{1}}^{p} \dx\right)^{\frac{1}{p}}\nonumber \\*[0.5ex]
		&\leqslant&\varepsilon\int_{B_{\tau_{2}}(x_{0})}\snr{V_{\snr{z_{0}}}(D\varphi_{1})}^{2} \dx+\frac{c\snr{B_{\rr}(x_{0})}}{\eps^{1/(p-1)}}\left(\rr^{m}\dashint_{B_{\rr}(x_{0})}\snr{f}^{m} \dx\right)^{\frac{p}{m(p-1)}},
	\end{eqnarray}
	where $c\equiv c(n,N,m)$ and we also used that $\rr/2\leqslant\tau_{2}\leqslant\rr$. Hence, proceeding as in the proof of \eqref{cccp}, we obtain that
	\begin{eqnarray*}
	 &&	\int_{B_{\tau_{1}}(x_{0})}\snr{V_{\snr{z_{0}}}(Du-z_{0})}^{2} \dx\\*[0.5ex]
	 &&\quad \leqslant c\int_{B_{\tau_{2}}(x_{0})\setminus B_{\tau_{1}}(x_{0})}\snr{V_{\snr{z_{0}}}(Du-z_{0})}^{2}+\left| \ V_{\snr{z_{0}}}\left(\frac{u-\ell}{\tau_{2}-\tau_{1}}\right)\ \right|^{2} \dx \nonumber \\*[0.5ex]
	 &&\qquad+\frac{c\snr{B_{\rr}(x_{0})}}{\eps^{1/(p-1)}}\left(\rr^{m}\dashint_{B_{\rr}(x_{0})}\snr{f}^{m} \dx\right)^{\frac{p}{m(p-1)}},
	\end{eqnarray*}
	with $c\equiv c(n,N,\lambda,\Lambda,p)$. Concluding as in the proof of \eqref{cccp}, we eventually arrive at \eqref{cccp 2}.
\end{proof}
We will also need the following result.
\begin{lemma}\label{l7}
	Under assumptions \eqref{assf}$_{1,2,3}$, \eqref{sqc} and \eqref{f}, let $u\in W^{1,p}(\Omega,\mathds{R}^{N})$ be a local minimizer of functional \eqref{main_fnc}.  For any $ B_\rr(x_{0}) \Subset \Omega $ and any $s \in (0,\infty)$ it holds that
	\begin{eqnarray}\label{p arm}
		\modulo{ \ \dashint_{B_\rr(x_{0})} \langle\snr{\dd u}^{p-2}\dd u,\dd \F \rangle \ \dx } &\leqslant& s\norma{\dd \F}_{L^\infty(B_\rr (x_{0}))} \left( \ \dashint_{B_\rr(x_{0})} \snr{\dd u}^p \ \dx \right)^\frac{p-1}{p} \nonumber\\*[0.5ex]
		&& +c\,\omega(s)^{-1}\norma{\dd \F}_{L^\infty(B_\rr (x_{0}))} \dashint_{B_\rr (x_{0})} \snr{\dd u}^p \ \dx \nonumber\\*[0.5ex]
		&& +c\,\norma{\dd \F}_{L^\infty(B_\rr (x_{0}))} \left( \rr^m \dashint_{B_\rr(x_{0})} \snr{f}^m \ \dx\right)^{1/m},
	\end{eqnarray}
	for any $\F \in C^\infty_0(B_\rr(x_{0}),\erN)$, with $c \equiv c(n,N,\Lambda,\lambda,p)$.
\end{lemma}
\begin{proof}
	Given the regularity properties of the integrand $F$, we have that a local minimizer $u$ of \eqref{main_fnc} solves weakly the following integral identity (see \cite[Lemma 7.3]{schm}) 
	\begin{equation} \label{l4 1}
	\int_\Omega \big[\langle \partial F(\dd u),\dd \F \rangle - f \cdot \F \big] \ \dx =0 \qquad \mbox{for all } \F \in C^\infty_0(\Omega, \erN).   
	\end{equation}
	Now, fix $\F \in C^\infty_0(B_\rr(x_{0}),\erN)$ and split
	\begin{eqnarray*}
		&&\modulo{ \ \dashint_{B_\rr(x_{0})} \langle\snr{\dd u}^{p-2}\dd u,\dd \F \rangle \ \dx } \\*[0.5ex]
		&&\qquad\stackrel{\eqref{l4 1}}{\leqslant}   \modulo{ \ \dashint_{B_\rr(x_{0})} \langle \partial F(\dd u) -\partial F(0)-\snr{\dd u}^{p-2}\dd u,\dd \F \rangle \ \dx } + \modulo{ \ \dashint_{B_\rr (x_0)} f \cdot \F \ \dx}\\*[0.5ex]
		&&\qquad =: \mbox{I}_1 + \mbox{I}_2.
	\end{eqnarray*}
	We begin estimating the first integral $\mbox{I}_1$. For $s \in (0,\infty)$ we get
	\begin{eqnarray}
		\mbox{I}_1 &\leqslant& \frac{\norma{\dd \F}_{L^\infty(B_\rr (x_{0}))} }{\snr{B_\rr (x_{0})}} \int_{B_\rr(x_{0}) \cap \{\snr{\dd u} \leqslant\omega(s)\}} \snr{\partial F(\dd u) -\partial F(0) -\snr{\dd u}^{p-2}\dd u} \ \dx \nonumber\\*[0.5ex]
		&&  + \frac{\norma{\dd \F}_{L^\infty(B_\rr (x_{0}))} }{\snr{B_\rr (x_{0})}} \int_{B_\rr(x_{0}) \cap \{\snr{\dd u} > \omega(s)\}} \snr{\partial F(\dd u) -\partial F(0) -\snr{\dd u}^{p-2}\dd u} \ \dx \nonumber\\*[0.5ex]
		&\leqslant& s\norma{\dd \F}_{L^\infty(B_\rr (x_{0}))}   \left( \ \dashint_{B_\rr(x_{0})} \snr{\dd u}^p\ \dx \right)^\frac{p-1}{p} \\*[0.5ex]
		&& +  c\,\omega(s)^{-1}\norma{\dd \F}_{L^\infty(B_\rr (x_{0}))} \dashint_{B_\rr(x_{0}) } \snr{\dd u}^{p} \ \dx \nonumber.
	\end{eqnarray}
	On the other hand, the integral $\mbox{I}_2$ can be estimated as follows
	\begin{eqnarray*}
		\mbox{I}_2 &\leqslant& \left(\rr^{m}\dashint_{B_\rr(x_{0})}\snr{f}^{m}\  \dx\right)^{1/m}\left( \ \dashint_{B_\rr(x_{0})}\modulo{\frac{\F}{\rr}}^{m'}\  \dx\right)^{\frac{1}{m'}}\nonumber \\*[0.5ex]
		&\leqslant&\left(\rr^{m}\dashint_{B_{\rr}(x_{0})}\snr{f}^{m}\  \dx\right)^{1/m}\left( \ \dashint_{B_{\rr}(x_{0})}\left| \ \frac{\F}{\rr} \ \right|^{p^*}\  \dx\right)^{\frac{1}{p^*}}\nonumber \\*[0.5ex]
		&\leqslant&\left(\rr^{m}\dashint_{B_{\rr}(x_{0})}\snr{f}^{m} \dx\right)^{1/m}\left( \ \dashint_{B_{\rr}(x_{0})}\snr{\dd \F}^{p} \dx\right)^{\frac{1}{p}}\nonumber \\*[0.5ex]
		&\leqslant& \norma{\dd \F}_{L^\infty(B_\rr(x_0))}\left(\rr^{m}\dashint_{B_{\rr}(x_{0})}\snr{f}^{m} \dx\right)^{1/m}.
	\end{eqnarray*}
	Combining the inequalities above we obtain \eqref{p arm}.
\end{proof}

In this setting the analogous result of Proposition \ref{l2} is the following one.

\begin{proposition}\label{p1}
	Under assumptions \eqref{assf}$_{1,2,3}$, \eqref{sqc} and \eqref{f}, let $u\in W^{1,p}(\Omega,\mathds{R}^{N})$ be a local minimizer of functional \eqref{main_fnc}. Then, for any $\chi \in (0,1]$ and any~$\tau_1 \in (0,2^{-10})$, there exists $\eps_2 \equiv \eps_2(\data,\chi,\tau_1) \in (0,1)$ such that if the smallness conditions
	\begin{equation} \label{p1 1}
	\chi \snr{(\dd u)_{B_\rr(x_0)}} \leqslant\mathcal{E}(u;B_\rr(x_0)), \quad \mbox{and} \quad \mathcal{E}(u;B_\rr(x_0)) \leqslant\eps_2,
	\end{equation}
	are satisfied on a ball $B_\rr(x_0) \subset \ern$, then
	\begin{equation}  \label{p1 2}
	\mathcal{E}(u;B_{\tau_1\rr}(x_0)) \leqslant c_1\tau_1^{\beta_1}\mathcal{E}(u;B_\rr(x_0)) +c_1 \left(\rr^m \dashint_{B_\rr(x_0)} \snr{f}^m \ \dx \right)^\frac{1}{m(p-1)},
	\end{equation}
	for any $\beta_1 \in (0,2\alpha/p)$, with $\alpha \equiv \alpha (n,N,p) \in (0,1)$ is the exponent in \eqref{p1 8}, and $c_1 \equiv c_1(\data,\chi)$.
\end{proposition}

\begin{proof}
	We adopt the same notations used in the proof of Proposition \ref{l2}. Let us begin noticing that condition \eqref{p1 1}$_1$ implies the following estimate
	\begin{equation}\label{p1 3}
    \dashint_{B_\rr}\snr{\dd u}^p \ \dx \leqslant c_\chi \mathcal{E}(u;B_\rr)^p \qquad \text{with}~c_\chi := 2^p(1+\chi^{-p}).
	\end{equation} 
	Indeed, by \eqref{eccess} and \eqref{p1 1}, we have
	\begin{eqnarray*}
		\dashint_{B_\rr}\snr{\dd u}^p \ \dx &\leqslant& 2^p\dashint_{B_\rr}\snr{\dd u-(\dd u)_{B_\rr}}^p \ \dx + 2^p\snr{(\dd u)_{B_\rr}}^p\\*[0.5ex]
		&\leqslant& 2^p \mathcal{E}(u;B_\rr)^p +\frac{2^p}{\chi^p} \mathcal{E}(u;B_\rr)^p.
	\end{eqnarray*}
	Consider now
	$$
	\kk := c_\chi \mathcal{E}(u;B_\rr) + \left( \frac{\rr^m}{\eps_3}\dashint_{B_\rr} \snr{f}^m \, \dx\right)^\frac{1}{m(p-1)} \qquad \mbox{and}~v_0 := \frac{u}{\kk},
	$$
	for $\eps_3 \in (0,1]$, which will be fixed later on. Applying \eqref{p arm} to the function $v_0$ yields  
	$$
	\modulo{ \ \dashint_{B_{\rr/2}(x_{0})} \langle\snr{\dd v_0}^{p-2}\dd v_0,\dd \F \rangle \, \dx } \stackrel{\eqref{p1 1}_2,\eqref{p1 3}}{\leqslant} c \norma{\dd \F}_\infty \left(s
	+\omega(s)^{-1}\eps_2 +\eps_3 \right).
	$$
	For any $\eps>0$ and $\vartheta \in (0,1)$ and let $\delta$ be the one given by \cite[Lemma 1.1]{dsv}. Then, up to choosing $s$, $\eps_2$ and $\eps_3$ sufficiently small, we arrive at
	$$
	c \left(s
	+\omega(s)^{-1}\eps_2 +\eps_3 \right) \leqslant\delta \|\dd \F\|_\infty^{p-1}.
	$$
	Then, Lemma 1.1 in \cite{dsv} implies
	\[
	\left( \ \dashint_{B_{\rr/2}} \snr{V(\dd v_0)-V(\dd h)}^{2\vartheta} \, \dx \right)^\frac{1}{\vartheta} \leqslant c \eps \  \dashint_{B_{\rr/2}}\snr{\dd u}^p \, \dx \stackrel{\eqref{p1 3},\eqref{p1 1}_2}{\leqslant} c \eps\eps_2^p,
	\]
	up to taking $\eps$ as small as needed. Now, denoting with $\mathfrak{h}_0:=h\kappa$, we have that
    \[
	\left( \ \dashint_{B_{\rr/2}} \snr{V(\dd u)-V(\dd \mathfrak{h}_0)}^{2\vartheta} \, \dx \right)^\frac{1}{\vartheta} \leqslant \eps\eps_2^p \kk^p.
    \]
	Now, we choose $\vartheta:=(\mathfrak{s})'/2$, with $\mathfrak{s}$ being the exponent given by \eqref{revh}. Note that by the proof of \eqref{revh} it actually follows that $\vartheta <1$. Thus, choosing $\eps\eps_2^p \kk^p \leqslant\tau_1^{2n+4\alpha}$ (where $\alpha \in (0,1)$ is given by \eqref{p1 8}) we arrive at
	$$
	\left( \ \dashint_{B_{\rr/2}} \snr{V(\dd u)-V(\dd \mathfrak{h}_0)}^{(\mathfrak{s})'} \, \dx \right)^\frac{1}{(\mathfrak{s})'} \leqslant c \, \tau_1^{n+2\alpha}.
	$$
	By H\"older's Inequality, we have that
	\begin{eqnarray}\label{p1 4}
		&& \dashint_{B_{\rr/2}} \snr{V(\dd u)-V(\dd \mathfrak{h}_0)}^{2}  \dx \\*[0.5ex]
		&&\quad\leqslant\left( \ \dashint_{B_{\rr/2} } \snr{V(\dd u)-V(\dd \mathfrak{h}_0)}^{(\mathfrak{s})'} \, \dx \right)^\frac{1}{(\mathfrak{s})'}\times \notag \\*[0.5ex]
		&&\qquad \times\left( \ \dashint_{B_{\rr/2}} \snr{V(\dd u)-V(\dd \mathfrak{h}_0)}^{\mathfrak{s}} \, \dx \right)^\frac{1}{\mathfrak{s}} \nonumber.
	\end{eqnarray}
	Hence, since by \eqref{equiv.1} $V(z) \approx \snr{z}^p$, an application of estimates \eqref{revh} and \eqref{p1 3} now yields
	\begin{eqnarray}\label{p1 5}
		\left( \ \dashint_{B_{\rr/2}} \snr{V(\dd u)}^{\mathfrak{s}} \, \dx \right)^\frac{1}{\mathfrak{s}} &\leqslant& c\left( \ \dashint_{B_{\rr/2}} \snr{\dd u-(\dd u)_\rr}^{p_2} \, \dx \right)^\frac{p}{p_2} + c \snr{(\dd u)_\rr}^p \nonumber\\*[0.5ex]
		&\leqslant& c \ \dashint_{B_\rr} \snr{\dd u}^p \dx + c \left( \rr^m \dashint_{B_\rr}\snr{f}^m \dx \right)^\frac{p}{m(p-1)} +c\snr{(\dd u )_\rr}^p\nonumber\\*[0.5ex]
		&\leqslant& c \, \mathcal{E}(u;B_\rr)^p + c \left( \rr^m \dashint_{B_\rr}\snr{f}^m \dx \right)^\frac{p}{m(p-1)},
	\end{eqnarray}
	with $c \equiv c(\data,\chi)$.
	
	On the other hand, by classical properties of $p$-harmonic functions, we have that
	\begin{equation}\label{p1 6}
     \left( \ \dashint_{B_{\rr/2}} \snr{V(\dd \mathfrak{h}_0)}^{\mathfrak{s}} \, \dx \right)^\frac{1}{\mathfrak{s}} \leqslant c \ \dashint_{B_\rr}\snr{\dd \mathfrak{h}_0}^p \dx \leqslant c \ \dashint_{B_\rr}\snr{ \dd u }^p \dx \leqslant c \, \mathcal{E}(u;B_\rr)^p.
	\end{equation}
	Hence, combining \eqref{p1 4}, \eqref{p1 5} and \eqref{p1 6}, we get that
	\begin{eqnarray}\label{p1 7}
     &&	\dashint_{B_{\rr/2}} \snr{V(\dd u)-V(\dd \mathfrak{h}_0)}^{2}  \dx \notag \\*[0.5ex]
     && \quad \leqslant c \, \tau_1^{n+2\alpha} \mathcal{E}(u;B_\rr)^p +c \, \tau_1^{n+2\alpha} \left( \rr^m \dashint_{B_\rr} \snr{f}^m \dx \right)^\frac{p}{m(p-1)}
	\end{eqnarray}
	Let us recall that, for any $\tau_1 \in (0,2^{-10})$,  given the $p$-harmonic function $\mathfrak{h}_0$ we have
	\begin{equation}\label{p1 8}
	\widetilde{\mathcal{E}}(\mathfrak{h}_0;B_{\tau_1 \rr})^2 \leqslant c \tau_1^{2\alpha } \kk^p, \qquad \alpha \equiv \alpha(n,N,p) \in (0,1).
	\end{equation}
	Moreover, using Jensen's Inequality we can estimate the following difference as follows
	\begin{eqnarray*}
		\snr{(\dd u)_{\tau_1 \rr}- (\dd u)_\rr} &\leqslant&  \left( \ \dashint_{B_{\tau_1\rr}} \snr{\dd u -(\dd u)_\rr}^p \dx \right)^\frac{1}{p}\\*[0.5ex]
		&\leqslant&  \tau_1^{-\frac{n}{p}} \left( \ \dashint_{B_\rr} \snr{\dd u -(\dd u)_\rr}^p \dx \right)^\frac{1}{p}  \\*[0.5ex]         &\stackrel{\eqref{eccess},\eqref{p1 1}_2}{\leqslant}&  \tau_1^{-\frac{n}{p}}\eps_2.
	\end{eqnarray*}
	Thus, up to taking $\eps_2$ sufficiently small, by the triangular inequality, we obtain that $\frac{1}{2}\snr{(\dd u)_{\tau_1 \rr}} \leqslant\snr{(\dd u)_\rr} \leqslant2 \snr{(\dd u)_{\tau_1 \rr}}$. Hence, \eqref{Vm} yield 
	$$
	\snr{V_{\snr{(\dd u)_{\tau_1 \rr}}}(\cdot)}^2 \approx \snr{V_{\snr{(\dd u)_\rr}} (\cdot)}^2,
	$$
	and
	$$
	\snr{V((\dd u)_{\tau_1 \rr})-V((\dd u)_{\rr})}^2 \approx \snr{V_{\snr{(\dd u)_{\rr}}}\big((\dd u)_{\rr}-(\dd u)_{\tau_1 \rr}\big)}^2.
	$$
	Then,
	\begin{eqnarray*}
		\mathcal{E}(u;B_{\tau_1 \rr})^p &\stackrel{\eqref{quad exc 1}}{\leqslant}& c\, \widetilde{\mathcal{E}}(u;B_{\tau_1 \rr})^2\\*[0.5ex]
		&\stackrel{\eqref{quad exc 2}}{\leqslant}& c \, \dashint_{B_{\tau_1\rr}} \snr{V(\dd u)-V((\dd u)_{\tau_1\rr})}^2 \dx\\*[0.5ex]
		&\leqslant& c \, \tau_1^{-n} \ \dashint_{B_{\rr/2}} \snr{V(\dd u)-V(\dd \mathfrak{h}_0)}^2 \dx \\*[0.5ex]
		&& + c \ \dashint_{B_{\tau_1\rr}} \snr{V(\dd \mathfrak{h}_0)-V((\dd \mathfrak{h}_0)_{\tau_1\rr})}^2 \dx\\*[0.5ex]
		&\stackrel{\eqref{quad exc 2}}{\leqslant}& c \, \tau_1^{-n} \ \dashint_{B_{\rr/2}} \snr{V(\dd u)-V(\dd \mathfrak{h}_0)}^2 \dx + c \ \widetilde{\mathcal{E}}(\mathfrak{h}_0,B_{\tau_1\rr}) \\*[0.5ex]
		&\stackrel{\eqref{p1 7},\eqref{p1 8}}{\leqslant}& c \, \tau_1^{2\alpha} \mathcal{E}(u;B_\rr)^p+ c  \left(\rr^m \dashint_{B_\rr}\snr{f}^m \dx \right)^\frac{p}{m(p-1)},
	\end{eqnarray*}
	and the desired estimate \eqref{p1 2} follows.
\end{proof}

\section{Proof of the main result}\label{proof_of_the_main}
This section is devoted to the proof of Theorem \ref{t1}. First, we prove the following proposition.

\begin{proposition}\label{loc_BMO_prop}
	Under assumptions \eqref{assf}$_{1,2,3}$, \eqref{sqc} and \eqref{f}, let $u\in W^{1,p}(\Omega,\mathds{R}^{N})$ be a local minimizer of functional \eqref{main_fnc}. Then, there exists~$\eps_* \equiv \eps_*(\data) >0$ such that if the following condition 
	\begin{equation}\label{main_thm_prop_1}
		\mathcal{E}(\dd u; B_r) + \sup_{\rr \leqslant r} \left( \rr^m \dashint_{B_\rr} \snr{f}^m \dx \right)^\frac{1}{m(p-1)} < \eps,
	\end{equation}
	is satisfied on $B_r\subset \Omega$, for some $\eps \in (0,\eps_*]$,  then
	\begin{equation}\label{main_thm_prop_2}
		\sup_{\rr \leqslant r} \mathcal{E}(\dd u; B_\rr) < c_3\, \eps,
	\end{equation}
	for $c_3 \equiv c_3(\data)>0$.
\end{proposition}
\begin{proof}
   For the sake of readability, since all balls considered in the proof are concentric to $B_r(x_0)$, we will omit denoting the center.
	
	Let us start fixing an exponent~$\beta \equiv \beta(\alpha,p)$ such that
	\begin{equation}\label{beta}
		0 < \beta < \min\{\beta_0,\beta_1\}=:\beta_m,
	\end{equation}
	where $\beta_0$ and$\beta_1$ are the exponents appearing in  Proposition \ref{l2} and Proposition \ref{p1}. Moreover, given the constant $c_0$ and $c_1$ from Proposition \ref{l2} and Proposition \ref{p1}, choose $\tau \equiv \tau(\data,\beta)$ such that
	\begin{equation}\label{tau}
		(c_0+c_1) \tau^{\beta_m -\beta} \leqslant\frac{1}{4}.
	\end{equation}
   With the choice of~$\tau_0$ as in~\eqref{tau} above, we can determine the constant~$\eps_0$ and~$\eps_1$ of Proposition~\ref{l2}. Now, we proceed applying Proposition~\ref{p1} taking~$\chi \equiv \eps_0$  and~$\tau_1$ as in~\eqref{tau} there. This determines the constant~$\eps_2$ and~$c_2$. We consider a ball~$B_r \subset \Omega$ such that
	\begin{equation}\label{bound_exc}
		\mathcal{E}(\dd u; B_r) <\eps_2,
	\end{equation}
	and 
	\begin{equation}\label{bound_data_f}
		\sup_{\rr \leqslant r} c_2 \left(\rr^m \dashint_{B_\rr} \snr{f}^m \dx\right)^\frac{1}{m(p-1)} \leqslant\frac{\eps_2}{4},
	\end{equation}
	where the constant $c_2 := c_1 +c_0$, with $c_0$ appearing in \eqref{l3 2} and $c_1$ in \eqref{p1 2}. In particular, see that by~\eqref{bound_exc} and~\eqref{bound_data_f} we are in the case when~\eqref{main_thm_prop_1} does hold true.	
	
	Now, we recall Proposition~\ref{p1}.  Seeing that~$\eqref{p1 1}_2$ is satisfied (being~\eqref{bound_exc}) we only check whether~$\eqref{p1 1}_1$ is verified too. If $\eps_0 \snr{(\dd u )_{B_r}} \leqslant\mathcal{E}(\dd u; B_r)$ is satisfied then we obtain from \eqref{p1 2}, with $\tau_1 \equiv \tau$ in \eqref{tau} that
	\begin{eqnarray}\label{main_thm_prop_3}
		\mathcal{E}(u;B_{\tau r}) &\leqslant& \frac{\tau^{\beta}}{4}\mathcal{E}(u;B_r) +c_2 \left(r^m \dashint_{B_r} \snr{f}^m \ \dx \right)^\frac{1}{m(p-1)} \nonumber\\*[0.5ex]
		&\leqslant& \frac{\tau^{\beta}}{4}\mathcal{E}(u;B_r) + \sup_{\rr \leqslant r}c_2 \left(\rr^m \dashint_{B_\rr} \snr{f}^m \ \dx \right)^\frac{1}{m(p-1)}\nonumber\\*[0.5ex]
		&\leqslant& \frac{\tau^{\beta}}{4}\mathcal{E}(u;B_r) + \frac{\eps_2}{4} \leqslant\eps_2,
	\end{eqnarray}
	where the last inequality follows from \eqref{bound_exc} and \eqref{bound_data_f}. If on the other hand it holds $\eps_0 \snr{(\dd u )_{B_r}} \geqslant  \mathcal{E}(\dd u; B_r)$, by Proposition \ref{l2}, then by \eqref{l2 4} or \eqref{l3 2} we eventually arrive at the same estimate \eqref{main_thm_prop_3}.
	
	Iterating now the seam argument we arrive at
	$$
	\mathcal{E}(\dd u ; B_{\tau^j r}) < \eps_2 \qquad \mbox{for any } j \geqslant 0,
	$$
	and the estimate 
	$$
	\mathcal{E}(u;B_{\tau^{j+1} r}) \leqslant \frac{\tau^{\beta}}{4}\mathcal{E}(u;B_{\tau^j r}) +c_2 \left((\tau^j r)^m \dashint_{B_{\tau^j r}} \snr{f}^m \ \dx \right)^\frac{1}{m(p-1)},
	$$
	holds true. By the inequality above we have that for any $k \geqslant 0$ 
	\begin{eqnarray*}
		\mathcal{E}(u;B_{\tau^{k+1} r}) &\leqslant&  \frac{\tau^{\beta(k +1)}}{4}\mathcal{E}(u;B_{r}) +c_2 \sum_{j=0}^{k} (\tau^\beta)^{j-k} \left((\tau^j r)^m \dashint_{B_{\tau^j r}} \snr{f}^m \ \dx \right)^\frac{1}{m(p-1)} \\*[0.5ex]
		&\leqslant&  \tau^{\beta(k +1)}\mathcal{E}(u;B_{r}) +c_2 \sup_{\rr \leqslant r} \left(\rr^m \dashint_{B_rr} \snr{f}^m \dx\right)^\frac{1}{m(p-1)}.
	\end{eqnarray*}
    Applying a standard interpolation argument we conclude that, for any~$t \leqslant r$, it holds
    \begin{equation}\label{ecc_dec_bmo}
    \mathcal{E}(\dd u, B_s) \leqslant c_3 \left(\frac{s}{r}\right)^\beta\mathcal{E}(\dd u, B_r) +c_3 \sup_{\rr \leqslant r} \left(\rr^m \dashint_{B_rr} \snr{f}^m \dx\right)^\frac{1}{m(p-1)}\,,
    \end{equation}
   where~$c_3 \equiv c_3(\data)$. The desired estimate \eqref{main_thm_prop_2} now follows.
\end{proof}

\begin{proof}[Proof of Theorem~{\rm\ref{t1}}]
   We proceed following the same argument used  in~\cite[Theorem~1.5]{kumi}. We star proving that, for any~$1\leqslant m < n$ and any~$\mathscr{O} \subset \Omega$, with positive measure, we have that
   \begin{equation}\label{jensen}
   \|f\|_{L^{m}(\mathscr{O})} \leqslant\left(\frac{n}{n-m}\right)^{1/m}\snr{\mathscr{O}}^{1/m-1/n}\|f\|_{L^{n,\infty}(\mathscr{O})}.
   \end{equation}
   Indeed, choose~$\bar{\lambda}$ which will be fixed later on. Then, we have that
   \begin{equation}\label{jensen 2}
      \|f\|_{L^m(\mathscr{O})}^m = m \int_0^{\bar{\lambda}}\lambda^m \snr{\{x \in \mathscr{O}:\snr{f}>\lambda\}}\frac{{\rm d}\lambda}{\lambda} +m \int_{\bar{\lambda}}^\infty\lambda^m \snr{\{x \in \mathscr{O}:\snr{f}>\lambda\}}\frac{{\rm d}\lambda}{\lambda}.
    \end{equation}
   The first integral on the righthand side of~\eqref{jensen 2} can be estimated in the following way
   \[
      \int_0^{\bar{\lambda}}\lambda^m\snr{\{x \in \mathscr{O}:\snr{f}>\lambda\}}\frac{{\rm d}\lambda}{\lambda} \leqslant\frac{\bar{\lambda}^m\snr{\mathscr{O} } }{m}.
   \]
   On the other hand, the second integral can be estimated recalling the definition of the $L^{n,\infty}(\mathscr{O})$-norm. Indeed,
   \[
    \int_{\bar{\lambda}}^\infty\lambda^m \snr{\{x \in \mathscr{O}:\snr{f}>\lambda\}}\frac{{\rm d}\lambda}{\lambda} \leqslant\|f\|_{L^{n,\infty}(\mathscr{O})}^n \int_{\bar{\lambda}}^\infty \frac{{\rm d}\lambda}{\lambda^{1+n-m}} \leqslant\frac{\|f\|_{L^{n,\infty}(\mathscr{O})}^n}{(n-m)\bar{\lambda}^{n-m}}.
   \]
   Hence, putting all the estimates above in~\eqref{jensen 2}, choosing~$\bar{\lambda}:=\|f\|_{L^{n,\infty}(\mathscr{O})}/\snr{\mathscr{O}}^{1/n}$, we obtain~\eqref{jensen}.

   \vspace{2mm}
   Now, recalling condition~\eqref{est_Lninfty} we have that
   \begin{eqnarray*}
   \left(\rr^m \dashint_{B_\rr}\snr{f}^m \, \dx\right)^{1/m} &\leqslant& \left(\frac{n}{n-m}\right)^{1/m}\snr{B_1}^{-1/n}\|f\|_{L^{n,\infty}(\Omega)}\notag\\*[0.5ex]
   &\stackrel{\eqref{f}}{\leqslant}& \left(\frac{4^{n/m}}{\snr{B_1}}\right)^{1/n}\|f\|_{L^{n,\infty}(\Omega)}
    \stackrel{\eqref{est_Lninfty}}{\leqslant} \eps_*\,,
   \end{eqnarray*}
  where~$\eps_*$ is the one obtained in the proof of Proposition~\ref{loc_BMO_prop}. From this it follows that, we can choose a radius~$\rr_1$ such that
   \begin{equation}\label{est_f}
   \sup_{\rr \leqslant\rr_1}c_2\left(\rr^m \dashint_{B_\rr(x)}\snr{f}^m \, \dx\right)^{1/m(p-1)} \leqslant\frac{\eps_*}{4c_3}.
   \end{equation}
   We want to show that the set~$\Omega_u$ appearing in~\eqref{t1 1} can be characterized by
   \[
   \Omega_u:= \left\{x_0 \in \Omega: \, \exists B_\rr(x_0)\Subset \Omega \, \text{with}~\rr \leqslant\rr_1 \, : \mathcal{E}(\dd u, B_\rr(x_0)) < \eps_*/(4c_3)\right\},
   \]
   thus fixing~$\rr_{x_0}:=\rr_1$ and~$\eps_{x_0}:= \eps_*/(4c_3)$.    We first star noting that the the set~$\Omega_u$ defined in~\eqref{t1 2} is such that~$\snr{\Omega \smallsetminus \Omega_u}=0$. Indeed, let us consider the set
   \begin{equation}\label{L_u}
    \mathscr{L}_u:= \left\{x_0 \in \Omega: \liminf_{\rr \to 0} \widetilde{\mathcal{E}}(u;B_\rr(x_0))^2=0\right\},
   \end{equation}
   which is such that~$\snr{\Omega \smallsetminus \mathscr{L}_u}=0$ by standard Lebesgue's Theory. Moreover, by~\eqref{quad exc 1} it follows that
   \[
       \mathscr{L}_u:= \left\{x_0 \in \Omega: \liminf_{\rr \to 0} \mathcal{E}(u;B_\rr(x_0))=0\right\},
   \]
  so that, $\mathscr{L}_u \subset \Omega_u$ and we eventually obtained that~$\snr{\Omega \smallsetminus \Omega_u}=0$. Now we show that~$\Omega_u$ is open. Let us fix~$x_0 \in \Omega_u$ and find a radius~$\rr_{x_0} \leqslant\rr_{1}$ such that
  \begin{equation}\label{est_ecc}
  \mathcal{E}(\dd u, B_{\rr_{x_0}}(x_0)) < \frac{\eps_*}{4c_3}.
  \end{equation}
  By absolute continuity of the functional~$\mathcal{E}(\cdot)$ we have that there exists an open neighbourhood~$\mathscr{O}(x_0)$ such that, for any~$x \in \mathscr{O}(x_0)$ it holds
 \begin{equation}\label{ecc_est}
    \mathcal{E}(\dd u, B_{\rr_{x_0}}(x)) < \frac{\eps_*}{4c_3} \quad \text{and}~B_{\rr_{x_0}}(x) \Subset \Omega.
  \end{equation}
  This prove that~$\Omega_u$ is open. Now let us start noting that~\eqref{est_f} and~\eqref{ecc_est} yield  that condition~\eqref{main_thm_prop_1} is satisfied  with~$B_r \equiv B_{\rr_{x_0}}(x)$. Hence, an application of Proposition~\ref{loc_BMO_prop} yields 
  \[
  \sup_{t \leqslant\rr_{x_0}}\mathcal{E}(\dd u, B_t(x)) < \eps_*,
  \]
 for any~$x \in \mathscr{O}(x_0)$. Thus concluding the proof.
\end{proof}


\vspace{5mm}
\begin{thebibliography}{99}
	
	\bibitem{AF84} {\sc E. Acerbi, N. Fusco}: Semicontinuity problems in the calculus of variations. {\it Arch. Rational Mech. Anal.} 86, no. 2, 125--145 (1984).
	\vs 
	
	\bibitem{AF87} {\sc E. Acerbi, N. Fusco}: A regularity theorem for minimizers of quasiconvex integrals. {\it Arch. Rational Mech. Anal.}  99 (1987), no. 3, 261--281.
	\vs

 \bibitem{AM01} {\sc E. Acerbi, G. Mingione}: Regularity results for a class of quasiconvex functionals with nonstandard growth. {\it Ann. Sc. Norm. Super. Pisa Cl. Sci.} (5) 30, 311--339, (2001).

 \vs
	
	\bibitem{BM84} {\sc J.~M. Ball, F. Murat}: $W^{1,p}$-quasiconvexity and variational problems for multiple integrals. {\it J. Funct. Anal.} 58, no. 3, 225--253 (1984). 
	\vs

\bibitem{bgik} {\sc M. B\"arlin, F. Gmeineder, C. Irving, J. Kristensen}: $\mathcal{A}$-harmonic approximation and partial regularity, revisited. \emph{Preprint} (2022). \href{https://arxiv.org/abs/2212.12821}{arXiv:2212.12821}
\vs

\bibitem{ba2} {\sc P. Baroni}: Riesz potential estimates for a general class of quasilinear equations, 
\emph{Calc. Var. \& PDE} 53, 803-846, (2015).
\vs

\bibitem{by} {\sc S.-S. Byun, Y. Youn}: Potential estimates for elliptic systems with subquadratic growth. \emph{J. Math. Pures Appl.} 131, 193-224, (2019).
\vs
 
    \bibitem{CFM98}	{\sc M. Carozza, N. Fusco, G. Mingione}: Partial regularity of minimizers of quasiconvex integrals with subquadratic growth. {\it Ann. Mat. Pura Appl.} (4) 175,  141--164 (1998).
	\vs

\bibitem {CiGA} {\sc A. Cianchi}: Maximizing the $L^\infty$-norm of the gradient of solutions to the Poisson equation. \emph{J. Geom. Anal.} 2, 499-515, (1992).

\vs

\bibitem{cm} {\sc A. Cianchi, V. G. Maz'ya}: Optimal second-order regularity for the $p$-Laplace system. \emph{J. Math. Pures Appl.} 132, 41--78, (2019).

\vs
 \bibitem{cm1} {\sc A. Cianchi, V. G. Maz'ya}: Global boundedness of the gradient for a class of nonlinear elliptic systems. \emph{Arch. Ration. Mech. Anal.} 212, 1, 129-177, (2014).
\vs

	\bibitem{dqc} {\sc C. De Filippis}: Quasiconvexity and partial regularity via nonlinear potentials. \emph{J. Math. Pures Appl.} (9) 163,  11--82 (2022).
	\vs
\bibitem{ds} {\sc C. De Filippis, B. Stroffolini}: Singular multiple integrals and nonlinear potentials. \emph{J. Funct. Anal.} 285(2), 109952, (2023).

 \vs
	
	\bibitem{de} {\sc L. Diening, F. Ettwein}: Fractional estimates for non-differentiable elliptic systems with general growth, \emph{Forum Math.} 20, 3, 523--556 (2008).
	\vs
	
	\bibitem{dlsv} {\sc L. Diening, D. Lengeler, B. Stroffolini, A. Verde}: Partial regularity for minimizers of quasi-convex functionals with general growth. \emph{SIAM J. Math. Anal.} 44, 5, 3594--3616 (2012).
	\vs
	
	\bibitem{dsv} {\sc L. Diening, B. Stroffolini, A. Verde}: The $\F$-harmonic approximation and the regularity of $\F$-harmonic maps. \emph{J. Differential Equations} 253, 1943--1958 (2012).
	\vs
\bibitem{dz} {\sc H. Dong, H. Zhu}: Gradient estimates for singular $p$-Laplace type equations with measure data. \emph{Preprint} (2021). \href{https://arxiv.org/pdf/2102.08584.pdf}{arXiv:2102.08584}


 \vs
 
	\bibitem{dumi} {\sc F. Duzaar, G. Mingione}: Regularity for degenerate elliptic problems via $p$-harmonic approximation. \emph{Ann. Inst. H. Poincar\'e Anal. Non Lin\'eaire}  21, 735--766 (2004).
	\vs

 \bibitem{dumi1} {\sc F. Duzaar, G. Mingione}: The $p$-harmonic approximation and the regularity of $p$-harmonic maps. \emph{Calc. Var. Partial Differential Equations} 20, 235--256, (2004).
	\vs

\bibitem{dust} {\sc F. Duzaar, K. Steffen}: Optimal interior and boundary regularity for almost minimizers to elliptic variational integrals. \emph{J. reine angew. Math.} 546, 73-138, (2002).
\vs
 
	\bibitem{Eva86} {\sc L.~C. Evans}: Quasiconvexity and partial regularity in the calculus of variations. {\it Arch. Rational Mech. Anal.} 95, no. 3, 227--252 (1986).
	\vs 
	
	\bibitem{gm86} {\sc M. Giaquinta, G. Modica}: Remarks on the regularity of the minimizers of certain degenerate functionals. \emph{Manuscripta Math.} 57, 55--99 (1986).
	\vs
	
	\bibitem{giu} {\sc E. Giusti}: \emph{Direct Methods in the calculus of variations}. {World Scientific Publishing Co., Inc., River Edge.} (2003).
	\vs

 \bibitem{gme} {\sc F. Gmeineder}: Partial regularity for symmetric quasiconvex functionals on $BD$. \emph{J. Math. Pures Appl.} 145, 83--129, (2021).

 \vs
 
\bibitem{gm1} {\sc F. Gmeineder}: The Regularity of Minima for the Dirichlet Problem on $BD$. \emph{Arch. Rational Mech. Anal.} 1099–-1171, 237(3), (2020).

\vs 

\bibitem{gk} {\sc F. Gmeineder, J. Kristensen}: Partial Regularity for BV Minimizers. \emph{Arch. Rational Mech. Anal.} 232, 1429--1473, (2019).

\vs

\bibitem{gkpq} {\sc F. Gmeineder, J. Kristensen}: Quasiconvex functionals of $(p,q)$-growth and the partial regularity of relaxed minimizers. \emph{Preprint} (2022). \href{https://arxiv.org/pdf/2209.01613.pdf}{arXiv:2209.01613}

 \vs

\bibitem{k} {\sc J. Kristensen}: On the nonlocality of quasiconvexity. {\it Ann. Inst. H. Poincar\'e Anal. Non Lin\'eaire}  16, 1, 1--13, (1999). 

 \vs
	\bibitem{KM07} {\sc J. Kristensen, G. Mingione}: The singular set of Lipschitzian minima of multiple integrals. {\it Arch. Ration. Mech. Anal.} 184, 341--369 (2007).
 
	\vs

 \bibitem{kumig} {\sc T. Kuusi, G. Mingione} Guide to nonlinear potential estimates. \emph{Bull. Math. Sci.} 4, 1-82, (2014).
 
 \vs

 \bibitem{kumil} {\sc T. Kuusi, G. Mingione}: Linear potentials in nonlinear potential theory. \emph{Arch. Ration. Mech. Anal.} 207,  215-246, (2013).

 \vs
	
	\bibitem{kumi} {\sc T. Kuusi, G. Mingione}: Partial regularity and potentials. \emph{J. \'Ecole Polytechnique Math.} 3, 309--363 (2016).
	\vs
	
\bibitem{kumi0} {\sc T. Kuusi, G. Mingione}: Vectorial nonlinear potential theory. \emph{J. Eur. Math. Soc.} 20, 929--1004 (2018).
          \vs
	
	\bibitem{ma3} {\sc P. Marcellini}: Approximation of quasiconvex functions, and lower semicontinuity of multiple integrals. \emph{Manuscripta Math.} 51, 1--3 (1985).
	\vs

\bibitem{ma2} {\sc P. Marcellini}: On the definition and the lower semicontinuity of certain quasiconvex integrals. {\it Ann. Inst. H. Poincar\'e Anal. Non Lin\'eaire}  3, nr. 5, 391-409, (1986).

\vs

\bibitem{ma5} {\sc P. Marcellini}: The stored-energy for some discontinuous deformations in nonlinear elasticity. \emph{Partial Differential Equations and the Calculus of Variations} vol. II, Birkh\"auser Boston Inc., (1989).

\vs
 
	\bibitem{Mor57} {\sc C.~B. Morrey}: Quasi-convexity and the lower semicontinuity of multiple integrals. {\it Pacific J. Math.} 2, 25--53 (1952).
	\vs 


\bibitem{musv} {\sc S. M\"uller, V. \v{S}ver\'ak}: Convex integration for Lipschitz mappings and counterexamples to regularity. {\em Ann. of Math. (2)} 157, 715--742, (2003). 

\vs

\bibitem{tn} {\sc Q.-H. Nguyen, N. C. Phuc}: A comparison estimate for singular $p$-Laplace equations and its consequences. \emph{Arch. Ration. Mech. Anal.} 247:49, (2003).

\vs
 
	\bibitem{Sch08} {\sc T. Schmidt}: Regularity theorems for degenerate quasiconvex energies with $(p,q)$-growth. {\it Adv. Calc. Var.} 1, no. 3, 241--270 (2008).
	\vs
	
	\bibitem{schm} {\sc T. Schmidt}:  Regularity of relaxed minimizers of quasiconvex variational integrals with $(p, q)$-growth, \emph{Arch. Rational Mech. Anal.} 193, 311--337 (2009).
	\vs
	
\end{thebibliography}
\end{document}